\documentclass[reqno,10pt]{amsart}

\usepackage{times}
\usepackage{microtype}
\usepackage[margin=1.1in]{geometry}
\usepackage[sort]{cite}
\usepackage{amsmath, amsthm, amssymb, amsfonts}
\usepackage{pdfsync}
\usepackage{parskip}

\usepackage{mathrsfs}  
\usepackage{mathtools}
\usepackage{mdframed}
\usepackage{tikz}

\numberwithin{equation}{section}

\usepackage[colorlinks=true, pdfstartview=FitV, linkcolor=blue,
citecolor=blue, urlcolor=blue]{hyperref}

\newtheorem{theorem}{Theorem}[section]
\newtheorem{prop}[theorem]{Proposition}
\newtheorem{definition}[theorem]{Definition}
\newtheorem{hypothesis}[theorem]{Hypothesis}
\newtheorem{lem}[theorem]{Lemma}

\theoremstyle{remark}
\newtheorem{remark}[theorem]{Remark}

\allowdisplaybreaks

\newcommand{\cL}{\mathcal L}
\newcommand{\cM}{\mathcal M}

\newcommand{\cS}{\mathcal S}

\def\R{\mathbb{R}}

\def\N{\mathbb{N}}

\def\sDp{\mathscr{D}_p}
\def\op{{\rm op}}
\def\gr{{\rm gr}}

\DeclareMathOperator{\diag}{diag}
\DeclareMathOperator{\dist}{dist}

\def\id{{\rm{id}}}

\def\div{{\rm div}}

\def\min{{\rm min}}
\def\max{{\rm max}}

\def\HQp{\mathring{H}^{1,p}_Q}

\def\HQpc{H^{1,p}_Q}

\begin{document}

\title[Degenerate parabolic $p$-Laplacians]{Degenerate parabolic
  $p$-Laplacian equations: existence, uniqueness, \\ and asymptotic
  behavior of solutions}

\author[D. Cruz-Uribe]{David Cruz-Uribe, OFS}
\address{David Cruz-Uribe, Dept. of Mathematics,
         The University of Alabama,
         Box 870350,
         Tuscaloosa, AL 35487-0350,
         USA}
\email{dcruzuribe@ua.edu}

\author[K. Moen]{Kabe Moen}
\address{Kabe Moen, Dept. of Mathematics,
         The University of Alabama,
         Box 870350,
         Tuscaloosa, AL 35487-0350,
         USA}
\email{kabe.moen@ua.edu}

\author[Y. Shao]{Yuanzhen Shao}
\address{Yuanzhen Shao, Dept. of Mathematics,
         The University of Alabama,
         Box 870350,
         Tuscaloosa, AL 35487-0350,
         USA}
\email{yshao8@ua.edu}




\thanks{The first and second authors are partially supported by Simons Foundation
  Travel Support for Mathematicians Grants.  The third author is partially supported by NSF grant DMS-2306991.}

\begin{abstract}
  In this paper we study the degenerate parabolic $p$-Laplacian,
  $ \partial_t u - v^{-1}\div(|\sqrt{Q} \nabla u|^{p-2} Q \nabla u){\color{blue}=0}$,
  where the degeneracy is controlled by a matrix $Q$ and a weight $v$.
  With mild integrability assumptions on $Q$ and $v$, we prove the
  existence and uniqueness of solutions {on any interval $[0,T]$}.  If we further assume
  the existence of a degenerate Sobolev inequality with gain, the
  degeneracy again controlled by $v$ and $Q$, then we can prove both
  finite time extinction and ultracontractive bounds.  Moreover, we
  show that there is equivalence between the existence of
  ultracontractive bounds and the weighted Sobolev inequality.
\end{abstract}

 \maketitle

\section{Introduction}

In this paper, we consider the existence, uniqueness, and asymptotic
behavior of weak solutions of the following degenerate parabolic $p$-Laplacian equation:
\begin{equation} \label{p-Lap-Dirichlet}
\begin{cases}
 \partial_t u - v^{-1}\div(|\sqrt{Q} \nabla u|^{p-2}  Q \nabla u)=0, &
 \text{in } \Omega \times [0,T],\\
 u(0) = u_0, &  \text{in } \Omega.
\end{cases}
\end{equation}
Throughout this paper, $\Omega$ is a bounded domain (i.e., an open and
connected subset) of  $\R^N$, with $N\geq 3$,  $Q: \Omega \to
\R^{N\times N}$ is a measurable,  symmetric matrix-valued function that
is positive definite a.e. and satisfies the degenerate ellipticity
condition
\begin{equation} \label{eqn:degen-ellipticity}
w(x) |\xi|^p \leq  |\sqrt{Q}(x) \xi|^p  \leq   v(x) |\xi|^p, \quad \xi\in \mathbb R^N,
\end{equation}
where $1<p<\infty$, and $v$ and $w$ are   weights (i.e.,  non-negative measurable
functions) that are in $L^1_{loc}(\Omega)$.
{See Remark~\ref{Rmk: boundary condition} for a brief
  discussion concerning boundary conditions for this equation.}

The literature on~\eqref{p-Lap-Dirichlet} and related PDEs is vast;
here we will only attempt to sketch some of the relevant history.  
When $Q\equiv I_{N\times N}$ and $v\equiv 1$,
equation~\eqref{p-Lap-Dirichlet} reduces to the 
parabolic $p$-Laplacian equation
\begin{equation}
\label{regularp-Lap-Dirichlet}
\partial_t u - \div(|  \nabla u|^{p-2}    \nabla u) =0.
\end{equation}
For decades, the parabolic $p$-Laplacian, along with its elliptic counterpart,
has been extensively studied, both for its
interesting mathematical structure but also because of its applications in fields like non-Newtonian flows, turbulent
flows in porous media and image processing.  The existence,
uniqueness, regularity, and asymptotic behavior of solutions to these
equations are well known:  see, for instance,~\cite{BonGrillo03, BonGrillo05, BonGrillo06,
  BonGrillo07, CipGri01, DiB83, DiB, DiBGiVe, Evans82, HerreroVaz81,
  Lewis, Lieber, Lions, StanVaz13, Tolksdorf, Trudinger67, Ural,
  Vaz06book}.  
In particular, we note that ultracontractivity bounds of
\eqref{regularp-Lap-Dirichlet}, i.e., inequalities of the form
$$
\|u(t)\|_\infty \leq \frac{C}{t^\beta} \|u_0\|_q^\gamma,
$$ 
for some $C,\beta, \gamma>0$ and $q\geq 1$, 
are equivalent to the existence  of Sobolev inequalities, such as
\begin{equation}
\label{Sobolev-Poincare unweighted}
\|u\|_{{2^*} } \leq C \|\nabla u\|_{2 },\quad  u\in C^{0,1}_c(\Omega),
\end{equation}
where $2^*=\frac{2N }{N-2}$ and $C^{0,1}_c(\Omega)$ is the collection of
Lipschitz functions of compact support in $\Omega$.

 
When $Q$ is a uniformly elliptic matrix, that is, when
\[ \lambda |\xi|^2 \leq \langle Q(x)\xi, \xi \rangle \leq \Lambda
  |\xi|^2, \qquad \xi \in \R^N, \] 
{it was shown in \cite{CipGri01, MR1911765} that the
  existence of a  Sobolev inequality implies that solutions satisfy  ultracontractivity bounds.
}
In this paper, we are interested in the case in
which $Q$ is degenerate: that is, the largest eigenvalue of $Q$ is
unbounded in $\Omega$ and the smallest eigenvalue of $Q$ can tend to
$0$.  (We note that the former case is referred to as singular by some
authors, who reserve degenerate for the latter case.  We will refer to
both cases simply as degenerate.)  More results are known in the
linear case, i.e., when $p=2$. In the elliptic case, when $w=v$ and
$w$ satisfies the so-called Muckenhoupt $A_p$ condition, these
equations were first studied in~\cite{MR0643158,MR0839035}.  The
corresponding parabolic equations were considered
in~\cite{MR1030914,MR0901091,MR0772255,MR0757584,MR0748366}; see
also~\cite{MR3261119,DavidRios08}.

{When the matrix $Q$ satisfies the degenerate ellipticity
condition,~\eqref{eqn:degen-ellipticity} has primarily been studied for
elliptic equations: see~\cite{DavidKabeVirginia,ChanWh,MR0805809,
  MR4224718,MR2574880,MR2204824,DavidScott21,DavidScottEmily18,
  DavidKabeScott,MR3388872,MR3369270,MR2906551,Ferrari,MR2305115,MR3922808}. } Not surprisingly,
weighted Sobolev and Poincar\'e inequalities play a very important
role in this work.  We note that beginning with the work of Sawyer and
Wheeden~\cite{MR2204824,MR2574880}, the study of these equations has
taken a more axiomatic approach.  In particular, it is frequently
assumed 
that a weighted Sobolev inequality exists without further specifying
conditions on $Q$, $v$ and $w$ for it to hold.

In this paper we will adopt a similar axiomatic approach.  We will
make a small set of assumptions on the weights $v$ and $w$
and the matrix $Q$ to prove existence, uniqueness and asymptotic
bounds.  We will show that the existence and uniqueness of a solution
to \eqref{p-Lap-Dirichlet} does {\em not} require that we assume a weighted Sobolev
inequality.  On the other hand, we will show that the  ultracontractive property of the solution
in some weighted Sobolev norm is equivalent to the existence of a  weighted
Sobolev inequality.  

The remainder of this paper is organized as follows.   In Section~\ref{Section: main
  results} we
state our main results.   We precede them by giving the 
definitions necessary to state our hypotheses and main results.  In
particular, we give a careful discussion of what we mean by
a weak solution to~\eqref{p-Lap-Dirichlet}.  As we were working on
this paper, we discovered that there are fundamentally different
understandings of solutions from the perspective of PDEs (represented
by the last author) and from the perspective of harmonic and
functional analysis (represented by the first two authors).  As this
initially caused a great deal of confusion, we think it is worthwhile
to discuss the issues involved with this question.  

In
Sections~\ref{section:sobolev} and~\ref{Section:semigroup} we gather
some technical results about degenerate Sobolev spaces and nonlinear
semigroup theory that is needed for our proofs.   In
Section~\ref{Section: Global weak solutions} we prove the existence
and uniqueness of solutions to~\eqref{p-Lap-Dirichlet}.  Finally, in
Sections~\ref{Section:strong-asymptotics} and~\ref{Section:asymptotics
  and Sobolev inequality} we prove the existence of asymptotic bounds
and show that they are equivalent to the existence of a weighted
Sobolev inequality.



\section{Statement of main results}\label{Section: main results}

In this section we will state our main results, Theorems~\ref{Thm:
  strong solution}--\ref{Thm: equivalence} below.  First, however, we must
state some basic definitions.  By a weight we mean a non-negative
function $v\in L^1_{loc}(\Omega)$.  For brevity, we will write $v$ for
the measure $v\,dx$: that is, given a measurable set
$E\subset \Omega$, we will write $v(E)= \int_E v\,dx$, and more generally
in integrals we will write $dv$ instead of $v\,dx$.  Let
$q\in [1,\infty)$; we define $L^q_v(\Omega)$ to be the space of all
measurable functions such that
$$
\|f\|_{L^q_v}:= \left( \int_\Omega |f|^q \, v \,dx \right)^{1/q} =  \left( \int_\Omega |f|^q \, dv \right)^{1/q}  <\infty.
$$
Similarly we define $L^\infty_v(\Omega)$ to be the space of all functions that are
essentially bounded with respect to the measure $v$.  When $q\in
(1,\infty)$, it is immediate that $L^q_v(\Omega)$  is a reflexive Banach space.

Let $\cM_N$ denote the collection of all real-valued, $N\times N$ matrices. The operator norm of a matrix $A\in \cM_N$ is defined by
$$
|A|_{\op}:=\sup\limits_{\xi\in \R^N, |\xi|=1}|A\xi|.
$$

Let $\cS_N$ denote the collection of $A\in \cM_N$ that are symmetric
(i.e., self-adjoint) and positive
semidefinite.
If $A\in \cS_N$, then it has $N$ non-negative eigenvalues, $\lambda_i$ with $1\leq i\leq N$, and
we have that $|A|_\op=\max_i \lambda_i$.

A matrix function is a map $A :\Omega \to \cM_N$; we say that it is measurable if each
component of $A$  is a measurable function. By a matrix weight we mean
a measurable matrix function $Q :\Omega \to \cS_N$ such that 
$|Q|_\op\in L^1_{loc}(\Omega)$.
A matrix weight $Q$ is diagonalizable by a measurable matrix function $U$
that is orthogonal: see~\cite[Lemma~2.3.5]{MR1350650}.  Consequently, each
eigenvalue (function) is measurable, and we have that each eigenvalue
$\lambda_i \in L^1_{loc}(\Omega)$ for $1\leq i\leq N$.  Moreover, we
can define powers of $Q$ as follows:  given $r\in \R$, if $D=\diag(\lambda_1,\ldots
\lambda_N)$ is such that $D = UQU^t$, then we define $Q^r=U^tD^rU$,
where $D^r=\diag(\lambda_1^r,\ldots,\lambda_N^r)$.  

Given a matrix weight $Q$, for $q\in [1,\infty)$ we define
$\cL^q_Q(\Omega)$ to be the space of all $\R^N$-valued vector fields
such that
$$
\| \vec{f} \|_{\cL^q_Q}
:=\left( \int_\Omega |\sqrt{Q} \vec{f}|^q \,   dx \right)^{1/q}<\infty.
$$
Again when $q \in (1,\infty)$, $\cL^q_Q(\Omega)$ is a reflexive Banach
space~\cite[Lemma~2.1]{DavidScottEmily18}.

Throughout this paper we will assume the following conditions on $Q$
and $v$.

\begin{hypothesis} \label{hyp:wt-matrix-hypotheses} The matrix weight
  $Q$ and the scalar weight $v$ in \eqref{p-Lap-Dirichlet} satisfy:
\begin{itemize}
\item $v\in L^1(\Omega)$.
\item $|\sqrt{Q}(x)|_{\op}^p\leq v(x)$ a.e.  
\item $|(\sqrt{Q})^{-1}|_{\op} \in L^{p'}_{loc}(\Omega).$
\end{itemize}
\end{hypothesis}

We define the scalar weight $w=|(\sqrt{Q})^{-1}|_{\op}^{-p}$;  then we have that $Q$ satisfies the two-sided degenerate ellipticity condition
$$w(x)^{2/p} |\xi|^2 \leq \langle Q(x) \xi,\xi \rangle \leq   v(x)^{2/p} |\xi|^2, \quad \ \xi\in \mathbb R^N,$$
which can be rewritten as
\begin{equation}
\label{Matrix weight}  w(x) |\xi|^p \leq  |\sqrt{Q}(x) \xi|^p  \leq   v(x) |\xi|^p, \quad \xi\in \mathbb R^N.
\end{equation}
The assumption ${ |(\sqrt{Q})^{-1}|_{\op} }\in L^{p'}_{loc}(\Omega)$
implies that $w^{1-p'}\in L^{1}_{loc}(\Omega)$. Further, 
$0<w(x)\leq v(x)<\infty$ a.e.

Our hypotheses on $Q$ and $v$ are essentially conditions imposed on
the largest and smallest eigenvalues of $Q$. Denote by
$\lambda_\max(x)$ and $\lambda_\min(x)$ the largest and smallest
eigenvalues of $Q(x)$. For example, if we choose
$v=|\sqrt{Q}|^p_{\op}=\lambda_\max^{p/2}$, then $v\in L^1(\Omega)$
provided $|Q|_\op=\lambda_\max\in L^{p/2}(\Omega)$.  Similarly, since
$|Q^{-1}|^{-1}_{\op}=\lambda_{\min}$, the assumption
$|(\sqrt{Q})^{-1}|_{\op}\in L^{p'}_{loc}(\Omega)$ is equivalent to
assuming $\lambda^{-1}_{\min}\in L^{p'/2}_{loc}(\Omega)$.

\begin{remark} \label{rem:hyp-wt}
  We gather together a few immediate consequences of
  Hypothesis~\ref{hyp:wt-matrix-hypotheses}.
  \begin{enumerate}
\item[(i)] Given a sequence $u_n\to u$ pointwise $v$-a.e. in $\Omega$,
  if we let 
$$
E=\{x\in \Omega: \lim\limits_{n\to \infty}u_n(x)\neq u(x)\},
$$
then $ v(E) =0$ implies that  $m(E)=0$, where $m$ is the Lebesgue measure
in $\R^N$. For otherwise, $v=0$ a.e. on $E$, which would contradict the fact that $0<w\leq v$ a.e.
Therefore,  $u_n\to u$ pointwise  a.e. in $\Omega$.

\item[(ii)] Since $v\in L^1(\Omega)$ we immediately have
$v(\Omega)<\infty$,   
which implies that
\begin{equation}\label{finite vol}
L^q_v(\Omega) \hookrightarrow L^r_v(\Omega) ,\quad 1\leq r \leq q \leq \infty.
\end{equation}
{
\item[(iii)] We have $\left(\sqrt{Q}\right)^{-1} \vec{f} \in L^{p^\prime}(\Omega )$ for all $ \vec{f} \in C_c(\Omega)$. Indeed, if $E={\rm supp}( \vec{f})$, then
\begin{equation}
  \int_\Omega \left|\left(\sqrt{Q}\right)^{-1}  \vec{f} \, \right|^{p'}\, dx
  \leq  \int_\Omega  \left| \left(\sqrt{Q} \right)^{-1} \right|_{\op}^{p'}| \vec{f}|^{p'}\, dx  
\leq  C \| \vec{f}\|^{p^\prime}_\infty\int_E  \left| \left(\sqrt{Q} \right)^{-1} \right|_{\op}^{p'} \, dx  
\leq   C\| \vec{f}\|^{p^\prime}_\infty.
\end{equation}
}
\end{enumerate}
\end{remark}

\medskip
 
In order to define the weak solutions of~\eqref{p-Lap-Dirichlet}, we
must first define the degenerate Sobolev space where they live.  The
Sobolev space $\HQpc(\Omega)$ is defined as the completion of
$C^{0,1}(\overline{\Omega})$ (that is, the space of Lipschitz
functions on $\overline{\Omega}$) with respect to the norm
$$
\|u\|_{\HQpc}:=\|u\|_{L^p_v} + \|\nabla u\|_{\cL^p_Q}.
$$
When $p \in (1,\infty)$, $\HQpc(\Omega)$ is a reflexive Banach space:
see~\cite[Lemma~2.3]{DavidScottEmily18}.  When $p=\infty$, this space
is defined analogously.  We note that, {\em a priori}, $\HQpc(\Omega)$
consists of equivalence classes of Cauchy sequences.  Since each
component of the norm is itself a norm, to each Cauchy sequence we can
associate a pair $(u,\vec{g})$, where $u \in L^p_v(\Omega)$ and
$\vec{g}\in \cL^p_Q(\Omega)$, but we need not have $\vec{g}=\nabla u$
even in the sense of weak (i.e., distributional) derivatives.
For a complete discussion of this phenomenon, see~\cite
{DavidScottEmily18}.  However, as we will show in
Section~\ref{section:sobolev} below,
Hypothesis~\ref{hyp:wt-matrix-hypotheses} actually guarantees that
$\HQpc(\Omega) \subset W^{1,1}_{loc}(\Omega)$ and that
$\vec{g}=\nabla u$.

The solution space $\HQp(\Omega)$ of \eqref{p-Lap-Dirichlet} is
defined to be the closure of $C^{0,1}_c(\Omega)$ in $\HQpc(\Omega)$.
It is immediate that when $p \in (1,\infty)$, $\HQp(\Omega)$ is also
a reflexive Banach space.

We can now define weak and strong solutions. Here and below,
we adopt the convention that if $u$ is a function defined on $\Omega
\times [0,T]$, we will write $u(t)$, $t\in [0,T]$, for the function
$u(\cdot,t)$. 


\begin{definition}\label{Def: Solution}
A measurable function $u:\Omega \times [0,T]\to \R$ is called a  weak
solution to  \eqref{p-Lap-Dirichlet} if the following hold:
\begin{itemize}
\item $u\in   W^{1,2}_{loc}((0,T); L^2_v(\Omega))\cap C([0,T];L^2_v(\Omega))$; 
\item $u(t)\in \HQp(\Omega)$ for a.e. $t\in (0,T)$; 
\item $\nabla u \in L^p((0,T); \cL^p_Q(\Omega))  $;
\item for any test function $\phi \in   C^\infty([0,T]; C^{0,1}_c(\Omega))$,
\begin{multline}
\label{Weak formulation}
\qquad \int_0^T \int_\Omega |\sqrt{Q} \nabla u|^{p-2} Q \nabla  u  \cdot \nabla \phi   \,dx \, dt   \\
=  \int_0^T\int_\Omega u(t)  \partial_t\phi (t)\, dv\,  dt +    \int_\Omega u_0 \phi(0)\, dv  - \int_\Omega u(T) \phi(T)\, dv;
\end{multline}

\item $u(t)\rightarrow u_0$ in $L^2_v(\Omega)$ as $t\rightarrow 0$.
\end{itemize}
\end{definition}

\begin{remark} \label{rem:weak-test-function}
By Proposition~\ref{Prop: density of L2} below and a density argument, one can easily check that, given any weak
solution $u$ of \eqref{p-Lap-Dirichlet}, \eqref{Weak formulation}
holds for all test functions $\phi \in   W^{1,2}((0,T);
L^2_v(\Omega))$ such that $\nabla \phi \in L^p((0,T);
\cL^p_Q(\Omega))$.
\end{remark}

\begin{definition} \label{Def:Strong-Solution}
A weak solution $u$ to  \eqref{p-Lap-Dirichlet} is called a
strong solution if $u$ further satisfies
\begin{itemize}
\item $u\in   W^{1,2}((0,T); L^2_v(\Omega)) $; 
\item $\nabla u \in L^\infty((0,T); \cL^p_Q(\Omega))$; 
\item for any $\phi  \in L^2 ((0,T); L^2_v(\Omega))$ such that  $\nabla \phi \in L^p((0,T); \cL^p_Q(\Omega))$,
\begin{equation}
\label{Strong formulation}
\int_0^T\int_\Omega \partial_t u(t) \phi (t)\, dv\,  dt   + \int_0^T \int_\Omega |\sqrt{Q} \nabla u|^{p-2} Q \nabla  u  \cdot \nabla \phi  \,dx \, dt  
=   0.
\end{equation}
\end{itemize}
\end{definition}

  \begin{remark}\label{Rmk: boundary condition}
    In stating equation~\eqref{p-Lap-Dirichlet} and defining its weak
    solution in Definition~\ref{Def:Strong-Solution}, we have not
    considered any boundary conditions.  Since the solution lives in
    $\HQp(\Omega)$, it would be natural to assume Dirichlet boundary
    conditions: that is,  $u=0$ on $\partial \Omega$.
    However, imposing boundary conditions for degenerate differential
    equations is a subtle issue. Indeed, even in the linear case,
    i.e. $p=2$, and $v\equiv 1$, Dirichlet boundary conditions will,
    in general, result in an overdetermined problem. Sufficient
    conditions for the well-posedness of \eqref{p-Lap-Dirichlet},
    assuming Dirichlet boundary conditions,
    are that the matrix $Q$ be smooth and be positive definite on $\partial\Omega$.
    See \cite{Fichera60, Oleinik64, Oleinik66} for more details.

    Following the theory of degenerate elliptic and parabolic
    equations as studied in, for instance,~\cite{MR0805809,
      MR4224718,MR2574880,DavidScott21,
      DavidKabeScott,MR3369270,MR2906551,MR2204824,MR2305115}, we want
    to impose the weakest hypotheses possible for a weak solution to
    exist.  In particular, we will allow $Q$ to be rough (i.e., with
    measurable coefficients) and to be non-positive definite along
    $\partial\Omega$.  We note that in many of the works cited,
    Definition~\ref{Def:Strong-Solution} would be regarded as the
    solution of the homogeneous Dirichlet problem, precisely because
    {the solution space is defined as the closure of the set of Lipschitz
    continuous functions with compact support in $H^{1,p}_Q(\Omega)$.}
    In the particular case that $v=w$ and $w$ satisfies the Muckenhoupt $A_2$
    condition, boundary conditions for the analogous elliptic operator
    have been studied in~\cite{MR0730093,MR0688024}.  We refer
    the interested reader to \cite{Amann20, GuidottiShao17, Shao18,
      YinWang04, YinWang07, ZhanFeng19, ZhanFeng21} for examples
    of nonlinear degenerate parabolic equations without boundary
    conditions or with partial boundary conditions.
  \end{remark}

 \begin{remark}
   A natural way to define the Dirichlet boundary condition $u=0$ on
   $\partial \Omega$ would be to use a trace operator: that is, a
   bounded linear map $T : A\rightarrow B$, where $A$ is some function
   space defined on $\Omega$ and $B$ is a function space defined on
   $\partial \Omega$.  However, when $A=\HQpc(\Omega)$, a trace
   operator $T$ cannot be defined in general.  Obviously, if it did
   exist, then by a density argument we would have that
   $T (\HQp(\Omega)) = \{0\}$.  But it is straightforward to construct
   a counter-example.  Let $\Omega$ be the unit ball $B(0,1)$ and
   define $d(x) = \dist(x, \partial \Omega)$.  Define
   $v=w=d^{\theta p}$ for some $\theta>1$ and $Q=v I_{N\times
     N}$. Then Hypothesis~\ref{hyp:wt-matrix-hypotheses} is satisfied.
   On the other hand, one can show that the function $f=d^{-\alpha}$
   with $\alpha=\theta-1+\frac{1-\Lambda}{p}$ satisfies
   $f\in \HQp(\Omega)$ for any $\Lambda>0$. If we choose
   $\Lambda\in (0,1)$ sufficiently small, then $\alpha>0$.  Consequently,
   $f(x)\to \infty$ as $|x|\rightarrow 1$. 

   The study of trace operators in weighted Sobolev spaces is still an
   active area of research.  For positive results see~\cite{Tyulenev,
     Tyulenev2, KosSotWang} and the references they contain.  Many of
   the positive results known assume that $v=w$ and $w$ is in the Muckenhoupt
   $A_p$ class.    In our example above, when $\theta>1$ implies that $\theta p /(1-p) \leq  -1$ and thus $v\notin A_p$. 
\end{remark}

\medskip

We can now state our first result, which gives the existence and
uniqueness of  weak solutions to \eqref{p-Lap-Dirichlet}: that
is, for every $T>0$, the function $u$ is a weak solution on $[0,T]$.
Moreover, we show that with additional hypotheses, weak solutions are
strong solutions.

\begin{theorem}
\label{Thm: strong solution}
Fix $1<p<\infty$ and assume the matrix $Q$ and weight $v$ satisfy
Hypothesis~\ref{hyp:wt-matrix-hypotheses}.  Then, for every
$u_0\in L^2_v(\Omega)$ {and $T>0$}, \eqref{p-Lap-Dirichlet} has a unique weak
solution $u$  {on $[0,T]$}.  Furthermore, if
$u_0\in \HQp(\Omega)\cap L^2_v(\Omega)$, then $u$ is a strong
solution.

If, in addition,
$u_0\in L^q_v(\Omega)\cap L^2_v(\Omega)$ {for $q\in [1,\infty]$}, then a weak solution $u$ is
$L^q_v$-contraction: i.e., for all $t>0$,
\begin{equation}
\label{L infty contraction}
\|u(t)\|_{L^q_v} \leq \  \|u_0\|_{L^q_v}.
\end{equation}

Finally, if $u_1,u_2$ are the weak solutions of \eqref{p-Lap-Dirichlet} with respect to initial data $u_{0,1},u_{0,2}\in L^2_v(\Omega)$, 
then for all $t>0$,
\begin{equation}
\label{comparison principle}
\int_\Omega (u_1(t) -u_2(t))^+\, dv \leq \int_\Omega (u_{0,1} -u_{0,2})^+\, dv.
\end{equation}
\end{theorem}

\begin{remark}
  By the uniqueness of solutions in Theorem~\ref{Thm: strong
    solution}, we know that{, for any $T>0$, there exists a unique   weak
  solution on $[0,T]$}.  Also note that any weak solution of
  \eqref{p-Lap-Dirichlet} is a strong solution on $[\tau, T]$ for all
  $\tau\in (0,T)$.
  {Indeed, given any weak solution $u$ on $[0,T]$, the uniqueness of a weak solution  
  implies that $u$ coincides with the mild solution with initial datum $u_0=u(0)\in L^2_v(\Omega)$.
  See Sections~\ref{Section:semigroup} and \ref{section:proof-global-strong}.
  Consider \eqref{p-Lap-Dirichlet} on $[\tau, T]$ with initial datum $u(\tau)\in \HQp(\Omega)\cap L^2_v(\Omega)$. 
  Then the additional regularity $\nabla u \in L^\infty((\tau,T); \cL^p_Q(\Omega))$ follows from 
  \eqref{L infty nabla u} below. }
\end{remark}

\bigskip

Finally, to prove the ultracontractivity of the solution, we will need
to assume that  a global weighted Sobolev inequality with gain holds.
In particular, we need the following.

\begin{hypothesis} \label{hyp:sobolev}
  Given $1<p<\infty$, there exist constants $M_p>0$ and
  $\sigma=\sigma_p>1$ such that matrix weight $Q$ and the weight $v$
  in~\eqref{p-Lap-Dirichlet} satisfy
\begin{equation}
\label{Sobolev ineq gain}
\|u\|_{L^{\sigma p}_v} \leq M_p \|\sqrt{Q}\nabla u\|_p ,\quad  u \in \HQp(\Omega).
\end{equation}
\end{hypothesis}

\begin{remark}
  By Hypothesis~\ref{hyp:wt-matrix-hypotheses}, $v(\Omega)<\infty$, so
  inequality \eqref{Sobolev ineq gain} implies a weighted Sobolev
  inequality without gain,
\begin{equation}
\label{Sobolev ineq}
\|u\|_{L^{p}_v} \leq C_p \|\sqrt{Q}\nabla u\|_p ,\quad  u \in \HQp(\Omega), 
\end{equation} 
albeit with a constant that depends on $v(\Omega)$. In
\cite{DavidScottEmily20}, it is shown that inequality \eqref{Sobolev
  ineq} is intimately connected to the solvability of the Dirichlet
problem for the degenerate $p$-Laplacian
$$
\begin{cases}\div(|\sqrt{Q}\nabla u|^{p-2}Q\nabla u)-\tau|u|^{p-2}u=f,
  & \text{in } \Omega, \\
  u=0, & \text{on }  \partial \Omega,
\end{cases}
$$
where $\tau>0$ and $f\in L^{p'}(\Omega)$.
\end{remark}

\begin{remark}
In general, little is known about when \eqref{Sobolev ineq gain}
holds.  One approach is through local two weight Sobolev inequalities
with weights $(w,v)$, namely,
\begin{equation}\label{localtwoweight}
  \left(\frac{1}{v(B)}\int_B|u|^{\sigma
      p} dv\right)^{\frac{1}{\sigma p}}
  \leq Cr(B) \left(\frac{1}{w(B)}\int_B|\nabla
    u|^{p}w\,dx\right)^{\frac{1}{p}}.
\end{equation}
If $w=|(\sqrt{Q})^{-1}|_{\op}^{-p}$, then by the left inequality in
\eqref{Matrix weight} we have that $w|\nabla u|^p\leq |\sqrt{Q}\nabla
u|^p$. Since there are no restrictions on the size of the ball $B$, we
have that \eqref{localtwoweight} implies \eqref{Sobolev ineq gain}.
Inequality \eqref{localtwoweight} has been investigated by Chanillo
and Wheeden \cite{ChanWh,MR0805809}; they  showed that \eqref{localtwoweight} holds under the following assumptions:
\begin{itemize}
\item $v$ is doubling and $w\in A_p(\Omega)$:
$$\left(\frac{1}{|B|}\int_B w\,dx\right)\left(\frac{1}{|B|}\int_B w^{1-p'}\,dx\right)\leq C$$
for all balls $B\subseteq \Omega$.
\item the pair $(w,v)$ satisfies the balance condition: there exists $q>p$ such that for every $0<s<1$ and ball $B\subset \Omega$
$$s\left(\frac{v(sB)}{v(B)}\right)^{1/q}\leq C \left(\frac{w(sB)}{w(B)}\right)^{1/p},$$
where $B=B(x,r)$ and $sB=B(x,sr)$ is the concentric ball with dilated radius.
\end{itemize} 
\end{remark}

\begin{remark}  We note in passing that the corresponding Poincar\'e inequality
without gain,
\begin{equation}
\label{Poincare no gain}
\left(\int_\Omega |u-u_\Omega|^pdv\right)^{1/p}\leq C\left(\int_\Omega |\sqrt{Q}\nabla u|^p\,dx\right)^{1/p}
\end{equation}
has been studied in \cite{CruzIsraMoen} and \cite{PerRela}. In particular, in \cite{PerRela}, the authors show that if $v$ and $Q$ satisfy
$$\sup_B\left(\frac{1}{|B|}\int_B dv\right)^{\frac1p}\left(\frac{1}{|B|}\int_B |\sqrt{Q}^{-1}|_{\op}^{-p'}\, {dx }\right)^{\frac1{p'}}<\infty$$
then the scalar Poincar\'e inequality
\begin{equation}
\label{scalarPoincare} \left(\int_\Omega |u-u_\Omega|^pdv\right)^{1/p}\leq C\left(\int_\Omega |\sqrt{Q}^{-1}|_{\op}^{-p}|\nabla u|^p\,dx\right)^{1/p}
\end{equation}
holds.  Inequality \eqref{scalarPoincare} in turn implies \eqref{Poincare no gain} since
$$|\sqrt{Q}^{-1}|_{\op}^{-p}|\nabla u|^p\leq |\sqrt{Q}\nabla u|^p.$$
Further, in \cite{DavidScottEmily18}, it is shown that \eqref{Poincare no gain} is equivalent to the solvability of the following Neumann boundary value problem  
\begin{align*}
\left\{\begin{aligned}
\div(|\sqrt{Q} \nabla u|^{p-2}  Q \nabla u)&= |f|^{p-2} f v &&\text{in}&& \Omega ,\\
\nu_\Omega \cdot Q\nabla u &=0 &&\text{on}&& \partial\Omega,
\end{aligned}\right.
\end{align*}
where $\nu_\Omega$ is outer unit normal of $\Omega$.
\end{remark}

\medskip

We can now state our results on the asymptotic behavior of weak solutions.

\begin{theorem} \label{Thm:strong asymptotics}
Suppose that Hypotheses~\ref{hyp:wt-matrix-hypotheses}
and~\ref{hyp:sobolev} hold.  
Given $q_0 \geq 1$ and  $q_0>q_c:=\sigma'(2-p)$, where $\sigma'=\sigma_p'$ is
the {H\"older conjugate} of the gain in the Sobolev inequality~\eqref{Sobolev ineq gain}.
Let $p_0=\max\{q_0,2\}$.
  Given $u_0\in   L^{p_0}_v(\Omega) $, let $u$ be the
  weak solution of \eqref{p-Lap-Dirichlet} {on $[0,T]$ for all $T>0$}. Then there exists a positive constant  $C$,
 depending only on $p$, $q_0$  and the constants $\sigma_p$ and $M_p$ in~\eqref{Sobolev ineq gain}, such that
 \begin{equation}
\label{asymptotics L2}
\| u (t)\|_{L^{\infty}_v} \leq C(\sigma_p,p,q_0,M_p) \frac{ \|u_0\|_{L^{q_0}_v}^\gamma }{t^\beta} ,
\end{equation}
where
\begin{align}\label{Def: beta gamma}
  \beta =  \frac{\sigma'}{\sigma'(p-2)+q_0}
  \quad \text{and } \quad
\gamma=  \frac{q_0}{\sigma'(p-2)  + q_0}. 
\end{align}
\end{theorem}

\begin{remark}
\label{Rmk: asymptotic}
\begin{itemize}
\item[]
\item[(i)] Our assumptions on $q_0$ guarantee that $\beta$ and $\gamma$ are positive. 
\item[(ii)] When $p\in [2,\infty)$, $q_c\leq 0$. In particular, we can always choose $q_0=1$ in \eqref{asymptotics L2}.
\item[(iii)] When {$p\in (2-2/\sigma' , 2)$}, $q_c<2$. In this case, we can always choose $q_0=2$ in \eqref{asymptotics L2}.
\end{itemize}
\end{remark}

\begin{theorem} \label{Thm: finite time extinction}
Suppose that Hypotheses~\ref{hyp:wt-matrix-hypotheses}
and~\ref{hyp:sobolev} hold and  $1<p<2$.
  Given $u_0\in   L^{m }_v(\Omega) $, where $m =\max\left\{2 , q_c  \right\}$ with $q_c:=\sigma'(2-p)$, let $u$ be the
   weak solution of \eqref{p-Lap-Dirichlet} {on $[0,T]$ for all $T>0$}.  Then there exists a finite time $T_0\in (0,\infty)$ such that $u(t,\cdot)= 0$  a.e. in $\Omega$ for all $t\geq T_0$.
\end{theorem}

%

The following theorem shows that the Sobolev inequality~\eqref{Sobolev ineq gain} is the weakest condition we can impose for the asymptotic estimate~\eqref{asymptotics L2}.

\begin{theorem} \label{Thm: equivalence}
Suppose that Hypothesis~\ref{hyp:wt-matrix-hypotheses} holds and  $p\in [2,\infty)$. 
Given $u_0\in   \HQp(\Omega) $, let $u$ be the
strong solution of \eqref{p-Lap-Dirichlet} {on $[0,T]$ for all $T>0$}.  
Then the following are equivalent:
\begin{itemize}
\item[(i)] the asymptotic estimate~\eqref{asymptotics L2} holds with $\beta$ and $\gamma$ defined by \eqref{Def: beta gamma};
\item[(ii)] the Sobolev inequality~\eqref{Sobolev ineq gain} holds;
\item[(iii)] for any $q_0\in [1,2)$, the following Nash inequality holds for all $u\in \HQp(\Omega)$ 
\begin{equation}
\label{G-N-ineq}
\|u\|_{L^2_v} \leq C \|\nabla u
\|_{\cL^p_Q}^{\frac{{\sigma p} (2-q_0)}{2\sigma p-2q_0 }} 
\|u\|_{L^{q_0}_v}^{\frac{\sigma q_0 (p-2)+ 2 {(\sigma-1)} q_0}{2\sigma p - 2q_0}}.
\end{equation}
\end{itemize}
Moreover, the constant $\sigma=\sigma_p$ in {\em (i)-(iii)} is the same.
\end{theorem}

\begin{remark}
  In the classical case (i.e., when $v=1$ and $Q=I$)
  Theorems~\ref{Thm:strong asymptotics}, \ref{Thm: finite time extinction}, and~\ref{Thm: equivalence} are all
  known.  See~\cite{BonGrillo07,MR2819280, CipGri01} and the references therein.
\end{remark}


\section{Degenerate Sobolev spaces}
\label{section:sobolev}

In this section we gather together some important properties of the
degenerate Sobolev spaces $\HQpc$ and $\HQp$ given the assumption of
Hypothesis~\ref{hyp:wt-matrix-hypotheses}.  We first show that these
spaces consist of functions in $W^{1,1}_{loc}(\Omega)$.  The proof of
the following result is almost identical to the proof of
Theorem~\cite[Theorem~5.2]{DavidKabeScott}.

\begin{prop} \label{prop:weak-derivs}
  For $1\leq p<\infty$, given a sequence $\{u_n\}_{n=1}^\infty$ of functions in
  $C^{0,1}(\overline{\Omega})$ that is Cauchy in $\HQpc$, then $u_n
  \rightarrow u$ in $L^p_v(\Omega)$, and $\nabla u_n \rightarrow
  \vec{g}$ in $\cL^q_Q(\Omega)$.  Moreover, $u \in
  W^{1,1}_{loc}(\Omega)$ and $\nabla u = \vec{g}$.
\end{prop}

The following results establish some basic properties of functions in
$\HQp(\Omega)$ similar to those of the classical Sobolev spaces.

\begin{prop}
\label{Prop: chain rule A}
Assume that $\phi\in C^1((0,\infty))\cap C([0,\infty))$ with $\phi(0)=0$. 
Given any $f\in \HQpc(\Omega)\cap L^\infty_v(\Omega)$ such that $f\geq c$ $v$-a.e. for some positive constant $c$, the function $\phi(f)\in \HQpc(\Omega)$. Moreover, $\nabla \phi(f) = \phi'(f) \nabla f $.
\end{prop}
\begin{proof}
The proof is similar to that of \cite[Lemma~2.18]{DavidScott21}.
\end{proof}

\begin{prop}
\label{Prop: chain rule A 2}
Assume that $\phi\in C^1((0,\infty))\cap C(\R)$ with $\phi(0)=0$. 
Given any $f\in \HQp(\Omega)\cap L^\infty_v(\Omega)$ such that $f\geq 0$ $v$-a.e. and $c>0$, the function $\left(\phi(f+c)- \phi(c) \right)\in \HQp(\Omega)$.
\end{prop}
\begin{proof}
Let $\{f_n\}_{n=1}^\infty\subset C^{0,1}_c(\Omega)$ be a sequence
that converges to $f$ in $\HQp(\Omega)$.  Clearly, 
$\{\phi(f_n+c)-\phi(c)\}_{n=1}^\infty\subset
C^{0,1}_c(\Omega)$. Furthermore  if we argue as in the proof of
\cite[Lemma~2.18]{DavidScott21}, it follows that
$
\phi(f_n+c)-\phi(c) \to \phi(f+c)- \phi(c)$ in $\HQp(\Omega)$.
\end{proof}

\begin{prop}
\label{Prop: cut-off A}
Assume that $f\in \HQp(\Omega)$.
For any $k\in \R$,   the functions $f_k^+=\max\{f,k\}$ and $f_k^-= \min\{f,k\}$
belong to $\HQp(\Omega)$. Moreover,
\[ \nabla f_k^+ = (\nabla f )\chi_{\{ f > k\}},
  \quad
  \nabla f_k^- = (\nabla f )\chi_{\{ f < k\}} \quad v \text{-a.e.}\]
\end{prop}

\begin{proof}
The proof is similar to that of \cite[Lemma~2.14]{DavidScott21}.
\end{proof}

\begin{prop}\label{Prop: density of ess bdd}
For $1\leq p<\infty$, $\HQp(\Omega)\cap L^\infty_v(\Omega)$ is dense in $\HQp(\Omega)$.
\end{prop}
\begin{proof}
  Fix $f\in \HQp(\Omega)$, and for each $k\in \N$, define
  $f_k:= \min\{\max\{f,-k\} , k\}$.  By Proposition~\ref{Prop: cut-off
    A}, $f_k \in \HQp(\Omega)\cap L^\infty_v(\Omega)$, so it remains
  to prove that $f_k \rightarrow f$ in $\HQp(\Omega)$.  We estimate
  each term in the norm separately.

  The first term is immediate: since $f_k \rightarrow f$  pointwise as
  $k\rightarrow \infty$, and $|f-f_k|^p \leq
  2^{p+1}|f|^p \in L_v^1(\Omega)$, by the dominated convergence
  theorem, $\|f-f_k\|_{L^p_v(\Omega)}\rightarrow 0$ as  $k\rightarrow
    \infty$.

    To estimate the second term, note that by Proposition~\ref{Prop: cut-off
    A}, $\sqrt{Q}\nabla f_k \rightarrow \sqrt{Q}\nabla f$ pointwise
  $v$-a.e.  Let $E\subset \Omega$ be the set where this convergence
  does not happen; then $v(E)=0$. 
  {An argument similar to that in Remark~\ref{rem:hyp-wt}(i) shows that $m(E)=0$.  Proposition~\ref{Prop: cut-off
    A}
  implies that $\sqrt{Q}\nabla f_k = \sqrt{Q}\nabla f \chi_{\{|f|<k\}}$
  almost everywhere.   Therefore,
  \[ |\sqrt{Q}(\nabla f-\nabla f_k)|^p =   |\sqrt{Q} (\nabla f -  \nabla f \chi_{\{|f|<k\}})|^p=  |\sqrt{Q} \nabla f  \chi_{\{|f|\geq k\}}|^p \leq  |  \sqrt{Q} \nabla  f|^p \]
  almost everywhere, and so again by the dominated convergence theorem,
   $\|\nabla f - \nabla f_k\|_{\cL_Q^p(\Omega)} \rightarrow
  0$ as $k\rightarrow \infty$. 
  }
  \end{proof}
  

\begin{prop}\label{Prop: density of L2}
$C^{0,1}_c(\Omega)$ is dense in  $\HQp(\Omega)\cap L^2_v(\Omega)$.
\end{prop}
\begin{proof}
  By Proposition~\ref{Prop: density of ess bdd}, it will suffice to show
  that for any $f\in \HQp(\Omega) \cap L^\infty_v(\Omega)$, there
  exist $\phi_k\in C^{0,1}_c(\Omega)$ such that $\phi_k \to f$ in
  $\HQp(\Omega)\cap L^2_v(\Omega)$.

By the definition of $\HQp(\Omega)$, there exists a sequence
$\psi_k\in C^{0,1}_c(\Omega)$ such that $\psi_k \to f$ in
$\HQp(\Omega)$ as $k\to \infty$.  By passing to successive
subsequences, we may assume that $\psi_k \rightarrow f$ pointwise
$v$-a.e., and $\sqrt{Q}\nabla \psi_k \rightarrow \sqrt{Q}\nabla f$
a.e.

Let $L=\|f\|_{L^\infty_v(\Omega)}$ and define
$\phi_k=\min\{\max\{\psi_k,-2L\} , 2L\} \in C^{0,1}_c(\Omega)$.   Then
$\phi_k \rightarrow f$ pointwise $v$-a.e., and if we argue as we did
in the proof of Proposition~\ref{Prop: density of ess bdd}, we get
that $\phi_k \rightarrow f$ in $\HQp(\Omega)$.

Finally, to see convergence in $L^2_v(\Omega)$, note that
$|\phi_k-f|^2 \leq 2^2(|\phi_k|^2+|f|^2) \leq 20L^2< \infty$.
By Hypothesis~\ref{hyp:wt-matrix-hypotheses}, $v\in L^1(\Omega)$, so
by the dominated convergence theorem, $\phi_k\rightarrow f$ in
$L^2_v(\Omega)$ as $k\rightarrow \infty$.
\end{proof}

\section{Nonlinear semigroup theory}
\label{Section:semigroup}

In order to find weak solutions to~\eqref{p-Lap-Dirichlet}, we will
apply nonlinear semigroup theory; our main result is
Proposition~\ref{Prop: maximal accretivee} below.   To state and prove
it, however, we must first introduce some abstract definitions and
results.   The main references for this appendix are \cite{Barbu, BenilanThesis, CraLig71,
  Showalter, MR2722295}.   Let $(\Omega,\mathcal{B},\mu)$ be a $\sigma$-finite
measure space, let $M(\Omega)$ be the set of all $\mu$-measurable
functions on $\Omega$, and for $q\in [1,\infty)$, let $L^q(\Omega)$
the corresponding $L^q$ space defined with respect to $\mu$.  
Let  $H$ be the Hilbert space $L^2(\Omega)$ and denote the inner
product by $\langle \cdot , \cdot \rangle_H$ and the norm by $\|\cdot\|_H$.
%


\medskip
\begin{definition}[Operator]
\label{Def: operator}
A map $A:H\to 2^H$ is called an operator on $H$. 
The {\em effective domain} of $A$ is the set
$$
D (A)=\{x\in H: A(x) \neq \emptyset \}.
$$
The {\em range} of $A$ is the set 
$$
{\rm Rng} (A) = \bigcup_{x \in H} A(x).
$$
The {\em graph} of $A$ is the set
$$
\gr(A)=\{(x,y)\in H\times H: \,  y\in A(x)\}.
$$
\end{definition}
An operator $A$ on $H$ can be considered as a (possibly) multi-valued function $A: D(A)\subset H \to H$.

\medskip
\begin{definition}[Maximal Monotone]
\label{Def: maximal monotone}
\begin{itemize}
\item[]
\item[(i)] An  operator
  $A: D(A)\to H$ is called {\em monotone} if
\begin{equation}
\label{A: monotone}
\langle y_1-  y_2 , x_1-x_2 \rangle_H \geq 0,  \quad (x_1,y_1),(x_2,y_2)\in  {\rm gr}(A). 
\end{equation}

\item[(ii)] An operator $A$ is called {\em maximal
    monotone} if $A$ is monotone and there is no monotone operator
  $\Psi$ such that ${\rm gr}(A)$ is a proper subset of
  ${\rm gr}(\Psi)$.
\end{itemize}
\end{definition}

\medskip
\begin{definition}[$m$-Accretivity] \label{Def: m-Accretivity}
  \ \\
  \begin{itemize}

\item[(i)] An operator
  $A: D(A)\to H$ is called {\em accretive} if for all
  $\lambda>0$
\begin{equation}
\label{A: contraction}
\|x_1-x_2 + \lambda(y_1-y_2)\|_H\geq \|x_1-x_2\|_H,\quad (x_1,y_1),(x_2,y_2)\in  {\rm gr}(A).
\end{equation}
\item[(ii)] An operator $A$ on $H$  is called {\em $m$-accretive} if $A$ is accretive and it satisfies the range condition
$$
{\rm Rng}(\id_H +\lambda A)=H, \quad \forall \lambda>0,
$$
where $id_H$ is the identity map on $H$.
\end{itemize}
\end{definition}

Given a function  $\Phi: X\subset H\to (-\infty,+\infty]$, we can extend $\Phi$ to $H$ by defining
$$
\Phi(u)=+\infty,\quad u\in H\setminus X.
$$
The {\em effective domain} of a function $\Phi:H\to (-\infty,+\infty]$ is the set 
$$
D(\Phi)=\{u\in H: \Phi(u)<+\infty\}.
$$
$\Phi$ is called {\em proper} if $D(\Phi)\neq \emptyset$ and {\em convex} if
$$
\Phi(\alpha x + (1-\alpha)y)\leq \alpha \Phi(x) + (1-\alpha)\Phi(y)
$$
for all $\alpha\in [0,1]$ and $x,y\in H$.

Given a proper, convex function  $\Phi:H\to (-\infty,+\infty]$,  the {\em subdifferential} $\partial \Phi$ of $\Phi$ is the operator $\partial\Phi: H\to 2^H$  defined by 
\begin{equation}
\label{Def: subgradient}
w\in \partial\Phi(z) \Longleftrightarrow \Phi(z)<+\infty  \quad \text{and} \quad \Phi(x)\geq \Phi(z) + \langle w , x-z \rangle_H ,\quad \forall x\in H.
\end{equation}
It is immediate from the definition of the subdifferential that
\begin{equation}
\label{min and subdifferential}
u  \text{ minimizes } \Phi \text{ within } H \quad \text{if and only if }\quad 0\in \partial \Phi(u).
\end{equation}

We now present several important facts concerning   subdifferentials and $m$-accretive operators.

\begin{prop}\label{Prop: m-accretive}
If $\Phi: H \to (-\infty,+\infty]$ is convex, proper and lower semicontinuous (l.s.c.), then $\partial \Phi$ is maximal monotone with $D(\partial\Phi)\subset D(\Phi)$ and $\overline{D(\partial \Phi)}=\overline{D(\Phi)}$.
\end{prop}



\begin{theorem}[Minty's Theorem]\label{Minty theorem}
Let $A$ be an accretive operator on $H$. Then $A$ is $m$-accretive if and only if $A$ is maximal monotone.
\end{theorem}

In order to construct a solution to
  \eqref{p-Lap-Dirichlet} by means of an implicit Euler scheme, we
  need the definition of an $\varepsilon$-approximate solution and a
  mild solution to the abstract Cauchy problem:
\begin{equation}
\label{abstract Cauchy pb}
\begin{split}
\partial_t u (t) + A u(t) &\ni 0 ,  \quad  t\in (0,T),\\
u(0)&= x_0.
\end{split}
\end{equation}

\begin{definition}[$\varepsilon$-Approximate Solution] \label{defn:approximate-soln}
  An $\varepsilon$-approximate solution $u (t)$ to \eqref{abstract
    Cauchy pb} on $[0,T]$ is a piecewise constant function defined by
  $u(0)=u_0$ with $\| x_0 -u_0 \|_H <\varepsilon$ and
\begin{align*}
u (t) = u_k \quad \text{ on } (t_{k-1}, t_k] ,\quad k=1,2,\cdots, n,
\end{align*}
where $0=t_0 <t_1 <\cdots < t_n =T$ with $t_k-t_{k-1} \leq \varepsilon$ and 
\begin{align*}
\frac{u_k -u_{k-1}}{t_k - t_{k-1}} + A u_k \ni 0, \quad k=1,2,\cdots, n.
\end{align*}
\end{definition}

\begin{definition}[Mild Solution] \label{defn:mild-soln}
  A mild solution to \eqref{abstract Cauchy pb} on $[0,T]$ is a
  function $u \in C([0,T];H)$    such that    for every $\varepsilon>0$
  there is an $\varepsilon$-approximate solution $v$ of
  \eqref{abstract Cauchy pb} on $[0,T]$ so that
{
\begin{equation*}  
\| u(t) - v(t) \|_H \leq  \varepsilon , \quad  t\in [0,T] 
\end{equation*}
Suppose that for every $\varepsilon>0$, \eqref{abstract Cauchy pb} has
$\varepsilon$-approximate solutions on $[0,T]$.  We say that the
$\varepsilon$-approximate solutions converge on $[0,T]$ to
$u \in C([0,T];H)$ as $\varepsilon \searrow 0$ if there exists a
function $\psi: [0,\infty)\to [0,\infty)$ with
$\lim\limits_{\varepsilon \searrow 0} \psi(\varepsilon)=0$ such that,
for every $\varepsilon$-approximate solution
$v$, \begin{equation} \label{mild solution converg} \| u(t) - v(t)
  \|_H \leq \psi(\varepsilon), \quad t\in [0,T] .
\end{equation}
}
\end{definition}

The following existence and uniqueness  theorem is due to Crandall and
Liggett~\cite{CraLig71}, and  B\'enilan~\cite{BenilanThesis}; see
also~\cite[Theorems A.27 and A.29]{MR2722295}.

\begin{theorem}\label{Benilan theorem}
Let $A $ be an $m$-accretive operator on $H$ and $T>0$.  
Then for every $x_0\in \overline{D(A)}$, there exists a unique mild
solution $u \in C([0,T];H)$ of \eqref{abstract Cauchy pb} such that
all $\varepsilon$-approximate solutions converges to $u$.
\end{theorem}

To state our final definition and result, we introduce some notation.  
We define the class
\begin{equation}
\label{class J0}
J_0=\{\text{convex, l.s.c. functions } j:\R\to [0, \infty] \text{ satisfying }j(0)=0\}.
\end{equation}
For $u,v \in M(\Omega)$, we write
\begin{equation}
\label{complete accretive cond}
u\ll v \quad \text{iff}\quad \int_\Omega j(u)\, d\mu \leq \int_\Omega j(v)\, d\mu ,\quad \text{for all }j\in J_0.
\end{equation}


\medskip
\begin{definition}
  \label{Def: Complete Accretivity}
An operator $A$ on $M(\Omega)$ is called {\em completely accretive} if 
$$
u_1 - u_2 \ll u_1 - u_2 + \lambda (v_1 - v_2)
$$
for all $\lambda>0$ and $(u_i,v_i)\in {\rm gr}(A)$ with $i=1,2$.
\end{definition}

{Note that when our Hilbert space is $L^2$, if we take $j(x)=x^2$
in~\eqref{complete accretive cond}, we get that completely accretive
implies accretive.}

To give a characterization of completely accretive operators, we define
\begin{equation}
\label{P0}
P_0 = \{f\in C^\infty(\R): \, 0\leq f^\prime\leq 1 , \, \text{{\rm supp}$(f^\prime)$ is compact, and } 0\notin \text{{\rm supp}$(f)$} \}.
\end{equation}

\begin{prop}\label{Prop: complete accretive}
An operator $A$ on  $L^q(\Omega)$ for
some $q \in (1,\infty)$ is completely accretive if and only if
$$
\int_\Omega f(u_1-u_2)(v_1-v_2) \,d\mu \geq 0 
$$
for all $f\in P_0$ and $(u_i,v_i)\in {\rm gr}(A)$, $i=1,\,2$.
\end{prop}

\medskip

To apply the above abstract theory to find a weak solution
to~\eqref{p-Lap-Dirichlet} we will consider the  functional
$\sDp: L^2_v(\Omega)\to [0,\infty]$ defined by
\begin{align*}
\sDp(u)=
\begin{cases}
\displaystyle \frac{1}{p}\int_{\Omega} |\sqrt{Q} \nabla u|^p \, dx,
&\text{ if } u \in   L^2_v(\Omega)\cap\HQp(\Omega),  \\
+\infty, & \text{ if }  u \in L^2_v(\Omega) \setminus \HQp(\Omega).
\end{cases}
\end{align*}
In addition, we define the operator $D_p$ on   $L^2_v(\Omega)$
as follows:  given $u \in L^2_v(\Omega) \cap \HQp(\Omega)$, $f\in D_p(u)$
if and only if $ f\in L^2_v(\Omega)$ and for any
$\phi\in C^{0,1}_c(\Omega)$,
\begin{equation} \label{eqn:subgradient0}
\int_\Omega  |\sqrt{Q} \nabla u|^{p-2} Q \nabla u \cdot \nabla \phi\,
dx
= \int_\Omega \phi f\,dv
\end{equation}
(That is, $u$ is a weak solution of the equation
$v^{-1}\div(|\sqrt{Q}\nabla u|^{p-2}Q\nabla u)= f$.)

\medskip

We can now give the result necessary to prove the existence of weak solutions.

\begin{prop}\label{Prop: maximal accretivee}
Fix $1<p<\infty$.  Then the  operator $\partial\sDp $ is an
$m$-completely accretive operator with
$D(\partial\sDp)\subset L^2_v(\Omega)\cap \HQp(\Omega)$, and
$\partial\sDp=D_p$.  Furthermore, if  $u\in  D(\partial\sDp)$ and
$f\in L^2_v(\Omega)$, 
then $f\in \partial \sDp(u) $ if and only if for all $\phi \in L^2_v(\Omega)\cap \HQp(\Omega)$,
\begin{equation}
\label{subgradient}
\int_{\Omega} f \phi\, dv =  \int_{\Omega}  |\sqrt{Q} \nabla u|^{p-2} Q \nabla u \cdot \nabla \phi \, dx. 
\end{equation}
\end{prop}

\begin{proof}
We will first show that $\partial \sDp$ is an $m$-accretive operator.  Since a
subdifferential is accretive (see~\cite[p.~157]{Showalter}), by
Minty's Theorem (Theorem~\ref{Minty theorem}) it will suffice to show that
$\partial \sDp$ is maximal monotone.  To do this, we will apply Proposition~\ref{Prop: m-accretive}.
By definition,  the effective domain of $\sDp$ is $D(\sDp)=
L^2_v(\Omega)\cap \HQp(\Omega)$, so $\sDp$ is  proper.  Since
$1<p<\infty$, $\sDp$ is convex.  Therefore, it remains to show that 
$\sDp$ is weakly lower semicontinuous (l.s.c) in $L^2_v(\Omega)$.

Assume that $u_n\to u$ in $ L^2_v(\Omega)$.
By passing to a subsequence, we may assume without loss of generality that
$\lim\limits_{n\to \infty} \sDp(u_n)$ exists and is finite.  In particular, each $u_n$ belongs to $L^2_v(\Omega)\cap \HQp(\Omega)$.  Since $u_n\to u$ in $L^2_v(\Omega)$, by passing to subsequence we may assume that $u_n\to u$ a.e on $\Omega$.  Ergoroff's Theorem now implies that for any $\varepsilon>0$, there exists $\Omega_\varepsilon\subset \Omega$ such that $m(\Omega\setminus \Omega_\varepsilon)<\epsilon$  and $u_n$ converges uniformly to $u$ on $\Omega_\varepsilon$. 
Since
$$\int_{\Omega_\varepsilon} |\sqrt{Q} \nabla u_n|^p \, dx\leq  \int_{\Omega} |\sqrt{Q} \nabla u_n|^p \, dx\leq C$$
we may use weak compactness (passing again to a subsequence) to conclude $\nabla u_n \rightharpoonup \vec{f}$ in
$\cL^p_Q(\Omega_\varepsilon)$ for some $\vec{f} \in \cL^p_Q(\Omega_\varepsilon)$. This implies that for any
$\phi \in C^1_c(\Omega_\varepsilon ;\R^N)$  
$$
\lim\limits_{n\to \infty} \int_{\Omega_\varepsilon} \nabla u_n \cdot \phi\, dx = \int_{\Omega_\varepsilon} \vec{f} \cdot \phi\, dx.
$$ 
Here we have used the fact that  $C^{0,1}_c(\Omega_\varepsilon;\R^N)\subset \cL^{p'}_{Q^{-1}}(\Omega_\varepsilon)\doteq (\cL^p_Q(\Omega_\varepsilon))^*$.
We claim
that $\nabla u=\vec{f}$ a.e. in $\Omega_\varepsilon$.
Indeed,
\begin{equation} \label{eqn:weak-estimate}
  \lim\limits_{n\to \infty} \int_{\Omega_\varepsilon} \nabla u_n \cdot \phi\, dx
  =  -\lim\limits_{n\to \infty} \int_{\Omega_\varepsilon}   u_n  \div \phi\, dx \\
= -  \int_{\Omega_\varepsilon}   u   \div \phi\, dx 
= \int_{\Omega_\varepsilon} \nabla u  \cdot \phi\, dx.
\end{equation}
The first equality follows from the inclusion $\HQpc(\Omega_\varepsilon)\subset W^{1,1}_{loc}(\Omega_\varepsilon)$.
The second equality is a direct consequence of the uniform convergence of $u_n\to u$ on $\Omega_\varepsilon$.
Therefore, $\nabla u_n$ converges to $\nabla u$ in the distributional
sense.  By Hypothesis~\ref{hyp:wt-matrix-hypotheses} and
Remark~\ref{rem:hyp-wt}, for any compact set $K\subset \Omega_\varepsilon$,
\[ \int_K |\nabla u_n|\,dx \leq   \bigg(\int_K w^{1-p'} \, dx \bigg)^{1/p'}
  \bigg(\int_K | \sqrt{Q} \nabla u_n|^p\,dv \bigg)^{1/p } <
  \infty. \]
Convergence in distributional sense,  
when restricted to locally integrable functions, coincides with the
usual convergence in $L^1_{loc}(\Omega_\varepsilon)$.   Therefore, $\nabla u\in
L^1_{loc}(\Omega_\varepsilon)$ and a standard approximation argument shows   that
$\vec{f}=\nabla u$ a.e. in $\Omega_\varepsilon$. The weak convergence implies
\begin{equation}
\label{liminf-epsilon}
\int_{\Omega_\varepsilon} |\sqrt{Q} \nabla u|^p \, dx\leq \liminf_{n\to \infty} \int_{\Omega_\varepsilon} |\sqrt{Q} \nabla u_n|^p \, dx \leq p\liminf_{n\to\infty}   \sDp(u_n).
\end{equation}
Finally we may take a nested sequence of domains $\Omega_{\varepsilon_1}\subset \Omega_{\varepsilon_2}\subset \cdots$ with $\Omega=\bigcup_k \Omega_{\varepsilon_k}$ to obtain
\begin{equation}
\label{liminf}
\sDp(u)=\lim_k\frac1p\int_{\Omega_{\varepsilon_k}} |\sqrt{Q} \nabla u|^p \, dx \leq \liminf\limits_{n\to \infty}\sDp(u_n).
\end{equation}
Therefore, we have shown that $\sDp$ is lower
semicontinuous, and so can conclude that $\partial\sDp$ is $m$-accretive.
Moreover, again by Proposition~\ref{Prop: m-accretive}, we have that
$\overline{D(\partial\sDp)} = \overline{D(\sDp)} =
L^2_v(\Omega) $. 

\medskip

We will now prove that $D_p=\partial\sDp$.  As the first step we  prove that $D_p$ is a completely accretive
  operator and so accretive.  From \cite{Simon} we have that for all $\xi,\zeta\in
  \R^N$, if $p\geq 2$ there exists $C>0$ such that
\begin{equation}
\label{Numerical inequality}
|\xi-\zeta|^p \leq   C(|\xi|^{p-2}\xi - |\zeta|^{p-2}\zeta)\cdot (\xi-\zeta) 
\end{equation}
and 
\begin{equation}
\label{Numerical inequality 2}
\big||\xi|^{p-2}\xi - |\zeta|^{p-2}\zeta \big| \leq \tilde{C}|\xi-\zeta|(|\xi|+|\zeta|)^{p-2};
\end{equation}
similarly, if $1<p<2$, there exists $\tilde{C}>0$ such that
\begin{equation}
\label{Numerical inequality 3}
\frac{|\xi-\zeta|^2}{(|\xi|+|\zeta|)^{2-p}} \leq   C(|\xi|^{p-2}\xi - |\zeta|^{p-2}\zeta)\cdot (\xi-\zeta) 
\end{equation}
and 
\begin{equation}
\label{Numerical inequality 4}
\big||\xi|^{p-2}\xi - |\zeta|^{p-2}\zeta \big| \leq \tilde{C}|\xi-\zeta|^{p-1}.
\end{equation}

Given any $f\in P_0$ (see~\eqref{P0})  and $v_i\in  D_p(u_i)$,
$i=1,\,2$,  it follows from \eqref{Numerical inequality} or
\eqref{Numerical inequality 3}   that
\begin{multline*}
\int_\Omega (v_1-v_2) f(u_1-u_2)\,dv \\
= \int_{\Omega} f^\prime(u_1-u_2) Q (|\sqrt{Q} \nabla u_1|^{p-2}  \nabla u_1 - |\sqrt{Q} \nabla u_2|^{p-2}   \nabla u_2 ) \cdot \nabla (u_1-u_2)\, dx \geq 0.
\end{multline*}
By Proposition~\ref{Prop: complete accretive}, we therefore have that $D_p$ is completely accretive.
{
If we  choose $j(x)=x^2$ in \eqref{complete accretive cond}, then Definition~\ref{Def: Complete Accretivity} implies that  $D_p$ is accretive.
}

To prove that $D_p$ is $m$-accretive, we will apply
Definition~\ref{Def: m-Accretivity}.
By~\cite[Lemma~IV.1.3]{Showalter} it will suffice to prove that
${\rm Rng}(I+ D_p)=L^2_v(\Omega)$. Given any $f\in L^2_v(\Omega)$, 
define the energy functional
\begin{equation} \label{eqn:energy-functional}
E_p(u)=\sDp(u) + \int_\Omega \left(\frac{u}{2} - f \right) u \, dv.
\end{equation}
We have shown that $\sDp(u)$ is l.s.c in $L^2_v(\Omega)$.  To see that
the functional $E_p$ is coercive, note that
\[ 
\int_\Omega (-fu)\, dv \geq - \|u\|_{L^2_v}  \| f \|_{L^2_v}  \geq - \varepsilon \|u\|_{L^2_v}^2 - C(\varepsilon) \|f\|_{L^2_v}^2.
\]
If we rearrange terms and choose $\varepsilon$ small, we get that the
second term in $E_p$ is uniformly bounded below; coercivity follows at once.
The
convexity of $E_p$ implies that it is also weakly l.s.c. in
$L^2_v(\Omega)$. Therefore,  by the direct method in the calculus
of variations (see~\cite[Theorem~4.6]{MR1962933}), $E_p$ has a minimizer
$u\in L^2_v(\Omega)\cap \HQp(\Omega)$ and its Euler-Lagrange equation shows
that $f\in I+D_p(u)$.

Since $D_p$ is $m$-completely
accretive, by {Theorem~\ref{Minty theorem} } it is maximal
monotone.  But, given any
$u,\,\phi\in L^2_v(\Omega)\cap \HQp(\Omega)$ and $f\in D_p(u)$, it follows from  Young's
inequality that
\begin{multline*}
  \int_\Omega   f (\phi-u)\,dv
  = \int_\Omega  |\sqrt{Q} \nabla u|^{p-2} Q \nabla u \cdot \nabla (\phi-u)\, dx  
 = \int_\Omega  |\sqrt{Q} \nabla u|^{p-2} Q \nabla u \cdot \nabla  \phi \, dx  -  p \sDp(u) \\
 \leq \frac{1}{p'}\int_{\Omega} |\sqrt{Q} \nabla u|^p \, dx +
 \frac{1}{p}\int_{\Omega} |\sqrt{Q} \nabla \phi|^p \, dx    -  p
 \sDp(u)
  =\sDp(\phi) - \sDp(u).
\end{multline*}
Hence, $f\in \partial\sDp(u)$, and so  $D_p \subset \partial \sDp$.  Thus,
by maximality, $D_p=\partial \sDp$.

Finally, we note that~\eqref{subgradient}
follows from~\eqref{eqn:subgradient0} in the definition of $D_p$ and a
density argument, since by Proposition~\ref{Prop: density of L2},
$C^{0,1}_c(\Omega)$ is dense in  $\HQp(\Omega)\cap L^2_v(\Omega)$.
\end{proof}


\section{Existence and uniqueness of solutions}\label{Section: Global weak solutions}
\label{section:proof-global-strong}

In this section, we prove Theorem~\ref{Thm: strong solution}.  We note
that the existence of a weak solution to \eqref{p-Lap-Dirichlet} can
be obtained by using an argument similar to that for the
parabolic $p$-Laplacian: see~\cite{Showalter}.
However, this definition of weak solution is somewhat different from
ours.  Moreover, this approach requires the Sobolev inequality
\eqref{Sobolev ineq}.    We will give a different  proof that does not require
a  Sobolev  inequality.  Moreover, as a consequence of our proof we
will derive many  energy estimates that will be useful for the asymptotic
analysis of \eqref{p-Lap-Dirichlet} in subsequent sections.

Our proof is organized as follows.  First, we will show that for any
$u_0\in L^2_v(\Omega)$,  \eqref{p-Lap-Dirichlet} has a {mild
solution}. The existence of such a solution is derived
using nonlinear semigroup theory. Second, we will show that mild
solutions satisfy~\eqref{L infty contraction}.   Third, we will show
that if  $u_0\in  L^2_v(\Omega)\cap \HQp(\Omega)$, then the mild
solution is a strong solution. Fourth, we will use  an approximation
argument to prove the existence of a weak solution when $u_0\in
L^2_v(\Omega)$.   Fifth, we will show that weak solutions are
unique.  Finally, we will show that the {comparison principle}~\eqref{comparison principle} holds for weak solutions.

\subsection{Existence of mild solutions}

{ We will use nonlinear semigroup theory to establish the
existence of a mild solution to \eqref{p-Lap-Dirichlet} by studying a
time-discretized problem.  By Definition~\ref{defn:mild-soln},
mild solutions are defined as the limit of any sequence of $\varepsilon$-approximate
solutions by implicit time discretization.  More
  precisely, since a mild solution does not depend on specific choice
  of discretization of a time interval $[0,T]$, we may divide $[0,T]$
  into $n$ equal subintervals.} Let $t_k=kT/n$ for $k=0,1,\cdots,n$.
Then the discretized problem to \eqref{p-Lap-Dirichlet} is
\begin{equation}
\label{p-Lap-Dirichlet-discretized}
\left\{\begin{aligned}
 \frac{n }{T } (u_{n,k} - u_{n,k-1}) + \partial \sDp ( u_{n,k}) & \ni 0  ;\\ 
u_{n,0} &=u_0  .
\end{aligned}\right.
\end{equation}
For any $\phi \in L^2_v(\Omega)\cap \HQp(\Omega)$, the first line of \eqref{p-Lap-Dirichlet-discretized} becomes
\begin{equation}\label{p-Lap-Dirichlet-discretized-weak}
\frac{T }{n  }\int_\Omega |\sqrt{Q} \nabla u_{n,k }|^{p-2} Q \nabla  u_{n,k } \cdot \nabla \phi \, dx  + \int_\Omega (u_{n,k} - u_{n,k-1})\phi\, dv=0.
\end{equation}
This degenerate elliptic equation has a weak solution; this has been
proved using the direct method of calculus variations as part of the
proof of Proposition~\ref{Prop: maximal accretivee}:
see~\eqref{eqn:energy-functional} and the subsequent discussion. 

Given this solution, we define the piecewise constant solution as
\begin{equation}
\label{piecewise constant fn}
u_n(t)=
\begin{cases}
u_0 ,\quad &   t=t_0 ; \\
u_{n,k}, &  t\in (t_{k-1},t_k],
\end{cases}
\end{equation}
that is, $u_n$ is an $\varepsilon$-approximate solution of \eqref{p-Lap-Dirichlet} for $\varepsilon>T/n$.
Proposition~\ref{Prop: maximal accretivee} and   Theorem~\ref{Benilan theorem} shows that \eqref{p-Lap-Dirichlet} has a unique mild solution $u \in C([0,T];L^2_v(\Omega)) $, which is 
the uniform limit of $u_n$, i.e. given any $\varepsilon>0$, 
for all sufficiently large $n$,
\begin{equation}
\label{mild sol def and converg}
\|u (t)-u_n(t)\|_{L^2_v} <\varepsilon, \quad t\in [0,T]
\end{equation}
Further, if we pass to a subsequence, we have that for all $t\in [0,T]$,  $u_n(t)$ converges to $u(t)$ pointwise $v$-a.e. in $\Omega$.

\subsection{Contractivity of mild solutions}
We will now show that $u$ satisfies \eqref{L infty contraction}. 
When $q=\infty$, pick  $u_0\in   L^\infty(\Omega)$.
Put $K=\|u_0\|_{L^\infty_v} $. Define the convex and l.s.c. function $j:\R\to [0,\infty)$ by
$$
j(x)=\max\{x,K\}-K.
$$
Then the complete accretivity of $\partial \sDp$ implies that
\begin{equation} \label{eqn:con-accret}
\int_\Omega j(u_{n,1} )\, dv \leq \int_\Omega j(u_0 )\, dv =0.
\end{equation}
To see this, first note that $0\in \partial \sDp(0)$ and
{$ u_0-u_{n,1} \in \frac{T}{n} \partial \sDp(u_{n,1})$}.  Using the fact that
{$\frac{T}{n}\partial \sDp$} is completive accretive and choosing
$(u_1,v_1)=(u_{n,1}, u_0-u_{n,1})$ and $(u_2,v_2)=(0,0)$ in
Definition~\ref{Def: Complete Accretivity}, we can conclude
that~\eqref{eqn:con-accret} holds.

Hence, $ u_{n,1}\leq K$ $v$-a.e.
Similarly, we can show that $ u_{n,1}\geq -K$ $v$-a.e.
Applying an analogous argument to every step of \eqref{p-Lap-Dirichlet-discretized} reveals that, for all $n$,
\begin{equation*}
\|u_n(t )\|_{L^\infty_v}   \leq \|u_0\|_{L^\infty_v},\quad t>0.
\end{equation*}
From  \eqref{mild sol def and converg}, we immediately infer that  $u$ is $L^\infty_v$-contraction.

When $q\in [1,\infty)$, we take 
$ 
j(x)=|x|^q.
$ 
A similar argument as above yields that  
\begin{equation}
\label{L q contraction discretized}
\|u_n(t )\|_{L^q_v}   \leq \|u_0\|_{L^q_v},\quad t>0.
\end{equation}
From the uniformly boundedness of $\|u_n(t )\|_{L^q_v}$, we infer that
$$
\|u (t )\|_{L^q_v} \leq \liminf\limits_{n\to \infty} \|u_n(t )\|_{L^q_v}\leq \|u_0\|_{L^q_v}.
$$


Choosing $\phi=u_{n,k}$ in \eqref{p-Lap-Dirichlet-discretized-weak} yields
\begin{equation} \label{Weak est 1}
  \notag  \frac{T}{n} \int_\Omega |\sqrt{Q} \nabla u_{n,k }|^p\, dx
  = \int_\Omega  u_{n,k-1} u_{n,k }\, dv  -\int_\Omega  u_{n,k }^2  \, dv 
\leq   \frac{1}{2} \left(\int_\Omega | u_{n,k-1 }|^2 \, dv  - \int_\Omega | u_{n,k }|^2 \, dv \right).
\end{equation}
Summing over $k=1,2,\ldots,n$ yields
$$
 \int_0^T \int_\Omega |\sqrt{Q} \nabla u_n|^p\, dx \, dt \leq      \frac{1}{2}  \int_\Omega| u_0|^2 \, dv.
$$
Hence,
\begin{equation}
\label{Weak est 2}
2\|\nabla u_n\|_{L^p((0,T);\cL^p_Q(\Omega))}\leq  \|u_0\|_{L^2_v}, \quad \|u_n(t)\|_{L^2_v} \leq  \|u_0\|_{L^2_v}.
\end{equation}

\medskip

\subsection{Existence of strong solutions}
\label{subsec:strong}
Now assume that    $u_0\in \HQp(\Omega)\cap L^2_v(\Omega)$.  We will
show that the mild solution found above is actually a strong solution.
When $1<p<2$, \eqref{mild sol def and converg} implies that
\begin{equation}
\label{mild sol Lp converg}
u_n  \to u  \quad \text{in }L^\infty((0,T); L^p_v(\Omega)) .
\end{equation}
When $p\in [2,\infty)$, \eqref{L q contraction discretized} implies
that the norms $\|u_n\|_{L^p((0,T);L^p_v(\Omega))}$ are uniformly bounded.
Hence, passing to a subsequence if necessary,  we have
\begin{equation}
\label{mild sol Lp converg 2}
u_n  \rightharpoonup u  \quad \text{in } L^p((0,T);L^p_v(\Omega)).
\end{equation}
{Applying an argument similar to that in  the proof of Proposition~\ref{Prop:
  maximal accretivee}, using \eqref{mild sol Lp converg} or
\eqref{mild sol Lp converg 2}, and again passing to a subsequence if
necessary, we have that }
\begin{equation}
\label{Weak convergence 0}
\nabla u_n \rightharpoonup \nabla u \quad \text{ in}\quad L^p((0,T);\cL^p_Q(\Omega)).
\end{equation}
Define 
$$
\tilde{u}_n(t)=
\begin{cases}
u_0, \quad & t=t_0, \\
\frac{n(t-t_{k-1})}{T}(u_{n,k}-u_{n,k-1})+u_{n,k-1},  & t\in (t_{k-1},t_k]. 
\end{cases}
$$
Then by \eqref{p-Lap-Dirichlet-discretized-weak}, we obtain
\begin{equation}\label{p-Lap-Dirichlet-discretized-weak-2}
\int_0^T \int_\Omega \partial_t \tilde{u}_n \phi\, dv \, dt + \int_0^T\int_\Omega |\sqrt{Q} \nabla u_n|^{p-2} Q \nabla  u_n \cdot \nabla \phi \, dx\,dt  = 0 
\end{equation}
for all valid test functions $\phi $ in \eqref{Strong formulation}.
{If we choose $\phi=u_{n,k}-u_{n,k-1}$ in
\eqref{p-Lap-Dirichlet-discretized-weak}, then   for all $t\in
(t_{k-1},t_k]$ we get
\begin{equation} \label{decreasing norm}
  \frac{T}{n}\int_\Omega |\partial_t \tilde{u}_n(t)|^2 \, dv
  =  \int_\Omega |\sqrt{Q}\nabla u_{n,k}|^{p-2} Q\nabla u_{n,k} \cdot \nabla(u_{n,k-1}- u_{n,k}) \,dx 
\leq  \sDp(u_{n,k-1})-\sDp(u_{n,k});
\end{equation}}
this inequality follows from  the definition of subdifferential and Proposition~\ref{Prop: maximal accretivee}.
Summing over $k=1,\ldots,n$ yields
\begin{equation}\label{Weak est 3}
\int_0^T \int_\Omega |\partial_t \tilde{u}_n|^2 \, dv\, dt  \leq     \sDp(u_0) - \sDp(u_n(T))  \leq \sDp(u_0).
\end{equation}
Hence, $\partial_t \tilde{u}_n$ is uniformly bounded in $L^2([0,T);L^2_v(\Omega)) $   with respect to $n$. 
By the construction of $\tilde{u}_n$, we readily see that
\begin{equation}\label{Converge tun}
\tilde{u}_n \to u \quad \text{ in} \quad L^\infty((0,T); L^2_v(\Omega)),
\end{equation}
and \eqref{Weak est 3} and \eqref{Converge tun} together imply that
\begin{equation}\label{Weak convergence 3}
\partial_t \tilde{u}_n \rightharpoonup  \partial_t u \quad \text{ in}\quad L^2((0,T); L^2_v(\Omega)).
\end{equation}
Now we choose $\phi=u_n-u$ in \eqref{p-Lap-Dirichlet-discretized-weak-2}. This yields
\begin{align*}
&\int_0^T \int_\Omega \partial_t \tilde{u}_n (u-u_n)\, dv \, dt\\
 & \qquad = \int_0^T\int_\Omega |\sqrt{Q} \nabla u_n|^{p-2} Q \nabla  u_n \cdot \nabla   (u_n-u) \, dx\,dt   \\
& \qquad = \int_0^T\int_\Omega \left(|\sqrt{Q} \nabla u_n|^{p-2} Q \nabla  u_n  - |\sqrt{Q} \nabla u|^{p-2} Q \nabla  u  \right)\cdot  \nabla (u_n-u) \, dx\,dt  + o_n(1), 
\end{align*}
where $o_n(1)\to 0$ as $n\to \infty$. 
The last step follows from~\eqref{Weak convergence 0}.
Together with \eqref{mild sol def and converg}, this implies
\begin{equation}\label{Weak estimate 4}
\lim\limits_{n\to \infty} \int_0^T\int_\Omega \left(|\sqrt{Q} \nabla u_n|^{p-2} Q \nabla  u_n  - |\sqrt{Q} \nabla u|^{p-2} Q \nabla  u \right)\cdot  \nabla (u_n-u) \, dx\,dt =0,
\end{equation}
and \eqref{Numerical inequality} and \eqref{Weak estimate 4}  imply that  when $p\geq 2$,
\begin{equation}\label{Strong convergence 1}
\nabla u_n\to \nabla u \quad \text{in } L^p((0,T);\cL^p_Q(\Omega)).
\end{equation}

When $1<p<2$, we use \eqref{Numerical inequality 3} and {H\"older's
inequality with exponent less than one.  This is a consequence of the
classical  H\"older's inequality:   fix $0<r<1$ and let
$\frac{r}{r-1}$ be the ``dual exponent''.  Then we have that
$$\int_0^T\int_\Omega |f(x,t)g(x,t)|\,dxdt\geq \|f\|_{L^r(\Omega\times
  (0,T))}\|g^{-1}\|_{L^{\frac{r}{1-r}}(\Omega\times (0,T))}^{-1}.$$}
This implies
\begin{align}\label{Numerical inequality 3-conclusion}
\notag & \int_0^T \int_\Omega \left(|\sqrt{Q} \nabla u_n|^{p-2} Q \nabla  u_n  - |\sqrt{Q} \nabla u|^{p-2} Q \nabla  u \right)\cdot  \nabla (u_n-u) \, dx\, dt\\
\notag &\qquad \geq  \int_0^T \int_\Omega  | \sqrt{Q} \nabla (u_n-u)|^2 (    |\sqrt{Q}\nabla u_n|+|\sqrt{Q}\nabla u|)^{p-2}\, dx\, dt \\
\notag & \qquad \geq  \|   \nabla (u_n-u)\|_{L^p((0,T);\cL^p_Q(\Omega))}^2 \| | \sqrt{Q} \nabla u_n|+| \sqrt{Q}\nabla u| \|_{L^p((0,T);L^p(\Omega))}^{p-2}\\
& \qquad \geq  C \|   \nabla (u_n-u)\|_{L^p((0,T);\cL^p_Q(\Omega))}^2  \left(\|   \nabla u_n\|_{L^p((0,T);\cL^p_Q(\Omega))}  + \|\nabla u \|_{L^p((0,T);\cL^p_Q(\Omega))} \right)^{p-2}   
\end{align}
for some $C$ independent of $n$.
Therefore, \eqref{Strong convergence 1} also holds when $1<p<2$.

By \eqref{Numerical inequality 2} or \eqref{Numerical inequality 4}
we have that
\begin{equation}
\label{Strong convergence 2}
|\sqrt{Q} \nabla u_{n }|^{p-2}\sqrt{Q}\nabla u_n \to |\sqrt{Q} \nabla u |^{p-2}\sqrt{Q}\nabla u \quad \text{in } L^p((0,T);L^{p^\prime}(\Omega)).
\end{equation}
On the other hand, inequality~\eqref{decreasing norm} implies that 
$$
\sDp(u_{n,k}) \leq \sDp(u_{n,k-1})\leq \sDp(u_0).
$$
If we use a similar argument to the proof of \eqref{Weak convergence
  0} and \eqref{Strong convergence 1}, we have that
\begin{equation}
\label{weak converg E measurable}
\nabla u_n \rightharpoonup \nabla u \quad \text{ in}\quad L^p(E;\cL^p_Q(\Omega )) 
\end{equation}
for any measurable $E\subseteq (0,T)$. A contradiction argument immediately implies that
\begin{equation}
\label{L infty nabla u}
\sDp(u(t)) \leq \sDp(u_0) ,\quad \text{for a.e. } t\in (0,T) \Longrightarrow \nabla u \in L^\infty((0,T); \cL^p_Q(\Omega )).
\end{equation}
{Indeed, if the set $E=\{t\in (0,T): \sDp(u(t)) > \sDp(u_0)\}$ had positive measure, then we would have 
$$
\| \nabla u \|_{L^p(E;\cL^p_Q(\Omega ))}   > \| \nabla u_0 \|_{L^p(E;\cL^p_Q(\Omega ))} .
$$ 
However, \eqref{weak converg E measurable} implies that
$$
\| \nabla u \|_{L^p(E;\cL^p_Q(\Omega ))} \leq \| \nabla u_n \|_{L^p(E;\cL^p_Q(\Omega ))}   \leq \| \nabla u_0 \|_{L^p(E;\cL^p_Q(\Omega ))} ,\quad \forall n \in \N,
$$
which is a contradiction.
}
For  any valid test function $\phi $ in \eqref{Strong formulation}, if
we let $n\to \infty$ in  \eqref{p-Lap-Dirichlet-discretized-weak-2},
\eqref{Weak convergence 3} and \eqref{Strong convergence 2}  we get
that
\begin{align*}
  \int_0^T\int_\Omega \partial_t u(t) \phi \, dv  dt + \int_0^T \int_\Omega |\sqrt{Q} \nabla u|^{p-2} Q \nabla  u  \cdot \nabla \phi  \,dx \, dt  
=0 .
\end{align*}
Therefore,  $u$ is a strong solution to \eqref{p-Lap-Dirichlet}.

\subsection{Existence of weak solutions}
Now fix $ u_0 \in L^2_v(\Omega)$.  We will use an approximation
argument to show that the mild solution found above is a weak
solution.  Recall that $C^{0,1}_c(\Omega) $ is dense in
$L^2_v(\Omega)$.  Fix a sequence
$u_{0,j} \in \HQp(\Omega) \cap L^2_v(\Omega)$ such that
$ u_{0,j} \to u_0$ in $L^2_v(\Omega)$; denote the corresponding strong
solution to \eqref{p-Lap-Dirichlet} with initial data $u_{0,j}$ by
$u_j$ and the corresponding piecewise constant solutions to
\eqref{p-Lap-Dirichlet-discretized} by $u_{n;j}$ with steps
$u_{n,k;j}$. 
{
We then get from equation~\eqref{p-Lap-Dirichlet-discretized} that 
$$
u_{n,k;j} + \frac{T}{n} \partial\sDp(u_{n,k;j}) \ni u_{n,k-1;j}.
$$
If we choose $\lambda=\frac{T}{n}$ in the contraction property~\eqref{A: contraction},
we get that
}
$$
\|u_{n ;j_1}(t)-u_{n ;j_2}(t)\|_{L^2_v} \leq \|u_{0,j_1}-u_{0,j_2}\|_{L^2_v};
$$
thus, $\{u_j\}_{j=1}^\infty$ is Cauchy in $C([0,T];L^2_v(\Omega))$. 
Therefore, for any $T>0$ there exists $u\in C([0,T];L^2_v(\Omega))$
such that
\begin{equation}
\label{Strong converg u 1}
u_j\to u \quad \text{in }  C([0,T];L^2_v(\Omega)).
\end{equation}
It is clear that $u$ is the unique mild solution of \eqref{p-Lap-Dirichlet} with initial datum $u_0$.
By \cite[Theorem~4.11]{Barbu2010}, for a.e. $t\in  (0,T]$, $u(t)\in \HQp(\Omega)\cap L^2_v(\Omega)$.
By the argument in Section~\ref{subsec:strong}, $u$ is a strong
solution on $[\tau,T]$ for all $\tau\in (0,T)$.  Hence, for any
$\phi \in L^p((\tau,T); \HQp(\Omega))\cap L^2 ((\tau,T);
L^2_v(\Omega)),$
\begin{equation}
\label{Weak formulation 0}
\int_\tau^T\int_\Omega \partial_t u(t) \phi (t)\, dv\,  dt   + \int_\tau^T \int_\Omega |\sqrt{Q} \nabla u|^{p-2} Q \nabla  u  \cdot \nabla \phi  \,dx \, dt  
=   0.
\end{equation}
In particular, \eqref{Weak formulation 0} holds for $\phi=u$.

\medskip

Next, we will prove that {$\nabla u\in L^p((0,T); \cL^p_Q(\Omega))$}. 
By~\cite[Theorem 1]{BenCra81} and   the accretivity of $\partial
\sDp$, we have that for sufficiently small $h$,
\begin{equation}
\label{time der est for u}
\frac{1}{h}\|u_j(t+h)-u_j(t)\|_{L^2_v} \leq \frac{2}{|p-1|t} \|u_{0,j}\|_{L^2_v} + o_h(1),
\end{equation}
where $o_h(1)\to 0$ as $h\to 0$.  For every $t\in (\tau,T)$, choose
$h>0$ so small that $t+h<T$.  Form a time discretization of $[0,T)$
with $n$ subintervals $[t_{k-1},t_k)$; without loss of generality, we
may always assume that $t$ and $t+h$ are points of the partition
$\mathcal{P}=\{0=t_0<t_1<\cdots <t_n=T\}$: given a partition, {one} can
always change the end points of the closest subintervals, as the
solution $u_j$ does not depend on the choice of time discretization.
{By~\eqref{p-Lap-Dirichlet-discretized-weak} we have that}
\begin{align*}
 (t_k-t_{k-1}) \int_\Omega |\sqrt{Q}\nabla u_{n,k;j}|^p\, dx 
\leq   \frac{1}{2} \int_\Omega  ( |u_{n,k-1;j}|^2 -|u_{n,k;j}|^2  ) \, dv.
\end{align*}
If we now  sum over all the subintervals $[t_{k-1}, t_k)$ contained in
$[t,t+h)$, then by the contraction property~\eqref{A: contraction},
{
\begin{align*}
  \int_t^{t+h} \int_\Omega |\sqrt{Q}\nabla u_{n;j}|^p\, dx\, ds
  &\leq  \frac{1}{2} \int_\Omega  ( |u_{n;j}(t)|^2 -|u_{n;j}(t+h)|^2  ) \, dv\\
& \leq  \frac{1}{2} \int_\Omega    |u_{n;j}(t)-u_{n;j}(t+h)| |u_{n;j}(t)+ u_{n;j}(t+h)|  \, dv \\
&\leq \frac{1}{2}\|u_{n;j}(t) - u_{n;j}(t+h)\|_{L^2_v}  ( \|u_{n;j}(t)\|_{L^2_v}  + \| u_{n;j}(t+h)\|_{L^2_v} ) \\
&\leq  \|u_{n;j}(t) - u_{n;j}(t+h)\|_{L^2_v}    \|u_{0;j} \|_{L^2_v}    \\
&\leq  C \|u_{n;j}(t) - u_{n;j}(t+h)\|_{L^2_v}   
\end{align*}
for some $C=C(\| u_0 \|_{L^2_v} )>0$.}
By~\eqref{finite vol}, \eqref{mild sol def and converg}, \eqref{Weak
  est 2}, and \eqref{time der est for u},  if we divide both sides by
$h$, we get
\begin{equation*}
   \frac{1}{h}\int_t^{t+h} \int_\Omega |\sqrt{Q}\nabla u_j|^p\, dx\, dt 
\leq        \frac{C }{  h }   \|u_j(t) - u_j(t+h)\|_{L^2_v	}    
\leq     \frac{   C \|u_{0,j}\|_{L^2_v}}{  (p-1) \tau  }     +o_h(1)
\leq  \frac{   C \|u_{0}\|_{L^2_v}}{  (p-1) \tau  }     +o_h(1) .
\end{equation*}
Therefore, if we take the limit as $h\rightarrow 0$, we get
\begin{equation}
\label{Est j 1}
\|\nabla u_j \|_{L^\infty ([\tau,T]; \cL^p_Q(\Omega))} <M.
\end{equation}
Here and in the sequel, $M\in (0,\infty)$ is a constant independent of
$j$.
It follows from~\eqref{Weak est 2} that
\begin{equation}
\label{Est j 2}
\|\nabla u_j \|_{L^p ((0,T); \cL^p_Q(\Omega))}  \leq M . 
\end{equation}
By the argument used to prove \eqref{Weak est 3} and by \eqref{Est j
  1}, we have that
\begin{equation*}
\|\partial_t u_j\|_{L^2( (\tau,T);L^2_v(\Omega))}<M.
\end{equation*}
Therefore,
possibly after passing to  a subsequence if necessary, we have that
\begin{align}\label{Weak converg u 1}
 \partial_t u_j \rightharpoonup \partial_t u   \quad  \text{in} \quad L^2_{loc}((0,T);L^2_v(\Omega)) .
\end{align}

If we use $u_j(\tau)$ as initial data, 
then~\eqref{Strong formulation} becomes
\begin{equation}\label{Weak formulation 2}
\int_\tau^T\int_\Omega \partial_t u_j(t) \phi (t)\, dv  dt   + \int_\tau^T \int_\Omega |\sqrt{Q} \nabla u_j|^{p-2} Q \nabla  u_j \cdot \nabla \phi  \,dx \, dt  
=   0.
\end{equation}
In particular, \eqref{Weak formulation 0} and \eqref{Weak formulation 2} hold  for $\phi=u$.
Since \eqref{Weak converg u 1} implies
$$
\lim\limits_{j\to \infty} \int_\tau^T\int_\Omega u (t) \partial_t u_j(t) \, dv  dt=\int_\tau^T\int_\Omega u (t)\partial_t u (t) \, dv  dt,
$$
it follows from \eqref{Weak formulation 0} and \eqref{Weak formulation
  2} that
\begin{equation*}
\lim\limits_{j\to \infty}   \int_\tau^T \int_\Omega |\sqrt{Q} \nabla u_j|^{p-2} Q \nabla  u_j \cdot \nabla u  \,dx \, dt 
=  \int_\tau^T \int_\Omega |\sqrt{Q} \nabla u |^{p-2} Q \nabla  u  \cdot \nabla u  \,dx \, dt.
\end{equation*}
If we now use $u_j$ as a test  function  in \eqref{Weak formulation 0}
and \eqref{Weak formulation 2}, we get
\begin{multline*}
\int_\tau^T \int_\Omega u_j \partial_t  (u -u_j) \, dv \, dt
=  \int_\tau^T\int_\Omega \left(|\sqrt{Q} \nabla u_j|^{p-2} Q \nabla u_j
  - |\sqrt{Q} \nabla u |^{p-2} Q \nabla  u   \right)\cdot  \nabla u_j \, dx\,dt   \\
= \int_\tau^T\int_\Omega \left(|\sqrt{Q} \nabla u_j|^{p-2} Q \nabla u_j
  - |\sqrt{Q} \nabla u |^{p-2} Q \nabla  u   \right)\cdot  \nabla (u_j -u)\, dx\,dt   + o_j(1), 
\end{multline*}
where $o_j(1)\to 0$ as $j\to \infty$.
Arguing as Section~\ref{subsec:strong} (see~\eqref{Strong convergence 1})  and using \eqref{Weak converg u 1}, we infer that
\begin{equation*}
\nabla u_j\to \nabla u \quad \text{in } L^p((\tau,T);\cL^p_Q(\Omega)) .
\end{equation*}
Because of \eqref{Est j 2}, we further have
\begin{equation}
\label{Property of u L2}
\nabla u_j \rightharpoonup
\nabla u   \quad \text{in } L^p((0,T); \cL^p_Q(\Omega)).
\end{equation}
On the other hand, \eqref{Weak formulation 0} is equivalent to
\begin{equation}
\label{Weak formulation 3}
\int_\tau^T \int_\Omega |\sqrt{Q} \nabla u|^{p-2} Q \nabla  u
\cdot \nabla \phi  \,dx \, dt  
=  \int_\tau^T\int_\Omega u(t)  \partial_t\phi (t)\, dv\,  dt
+    \int_\Omega u(\tau) \phi(\tau)\, dv  - \int_\Omega u(T) \phi(T)\, dv .
\end{equation}
By \eqref{Property of u L2}, if take the limit as $\tau\to 0^+$ in
\eqref{Weak formulation 3},  we get that $u$ is a weak solution.



\subsection{Uniqueness of weak solutions}
To see that weak solutions to~\eqref{p-Lap-Dirichlet} are unique, 
assume to the contrary that for some initial data $u_0 \in
L^2_v(\Omega)$,  there exist two   weak solutions $u, \tilde{u}$. 
Fix $0<t_1<t_2<T$; by Definition~\ref{Def: Solution} we have that 
{$u,\tilde{u}\in W^{1,2} ([t_1,t_2]; L^2_v(\Omega))$} and $\nabla u  , \nabla \tilde{u} \in L^p((0,T); \cL^p_Q(\Omega))$. 
{If we consider the \eqref{p-Lap-Dirichlet} on
  $[t_1,t_2]$, then  the function  
$$
\phi(t)=\ (u -\tilde{u} )(t)
$$
has the regularity given in Remark~\ref{rem:weak-test-function} and thus
is a valid test function in \eqref{Weak formulation} on $[t_1,t_2]$.}
If we use this as the test function against both $u$ and $\tilde{u}$,
subtract the resulting equations, and apply the fundamental theorem of
calculus and Fubini's theorem, we get
\begin{align*}
  & \int_{t_1}^{t_2} \int_\Omega (|\sqrt{Q} \nabla u|^{p-2} Q \nabla  u
  -|\sqrt{Q} \nabla \tilde{u}|^{p-2} Q \nabla  \tilde{u}  )(t,x) \cdot
  \nabla  (u-\tilde{u}   )(t,x)  \,dx \, dt \\
  & \qquad = {
  \int_{t_1}^{t_2} \int_\Omega (u-\tilde{u})
    (\partial_t u - \partial_t \tilde{u}) dv\,dt
    }
    + \int_\Omega (u(t_1)-\tilde{u}(t_1))^2 dv
    - \int_\Omega (u(t_2)-\tilde{u}(t_2))^2 dv \\
  & \qquad = -\frac{1}{2}\int_{t_1}^{t_2} \int_\Omega \partial_t\big( (u -\tilde{u}
    )^2\big)(t,x)\, dv\, dt.
\end{align*}
Here we have used the fact that $(u-\tilde{u})^2$ is absolutely continuous on $[t_1,t_2]$ with respect to $L^1_v(\Omega)$.
Further, it follows from \eqref{Numerical inequality} and \eqref{Numerical inequality 3} that
\[ 0 \leq  \int_{t_1}^{t_2} \int_\Omega (|\sqrt{Q} \nabla u|^{p-2} Q \nabla  u
  -|\sqrt{Q} \nabla \tilde{u}|^{p-2} Q \nabla  \tilde{u}  )(t,x) \cdot
  \nabla  (u-\tilde{u}   )(t,x)  \,dx \, dt. \]
If we combine these two estimates, we get that
\[ \int_\Omega (u(t_2)-\tilde{u}(t_2))^2 \,dv \leq \int_\Omega
  (u(t_1)-\tilde{u}(t_1))^2 \,dv. \]
Since $u,\tilde{u}\in C([0,T]; L^2_v(\Omega))$ and
$u(0)=\tilde{u}(0)$,  if we take the limit as $t_1\rightarrow 0$ we
get that for every $t_2 \in (0,T]$,
\[ \int_\Omega (u(t_2)-\tilde{u}(t_2))^2 \,dv \leq 0.  \]
Hence, we must have that $u(t)=\tilde{u}(t)$ almost everywhere in
$\Omega$ for all $t\in [0,T]$.

\subsection{Comparison principle for weak solutions}
{To verify \eqref{comparison principle}, we denote the solutions of \eqref{p-Lap-Dirichlet-discretized} with initial data $u_{0,i}$, $i=1,\,2$, by $u_{n,k;i}$. 
Since $\partial\sDp$ is completely accretive, if we apply
Definition~\ref{Def: Complete Accretivity} to the pairs
\[ (u_{n,k;i},u_{n,k-1;i}-u_{n,k;i}) \in {\rm gr}(\textstyle\frac{T}{n} \partial
\sDp),  \qquad i=1,\,2, \, {\forall k=1,\cdots,n,} \]
and if we take $\lambda=1$ in {Definition~\ref{Def: Complete Accretivity}}, we get
$$
u_{n,k;1}-u_{n,k;2}\ll u_{n,k-1;1}-u_{n,k-1;2},\quad   {\forall k=1,\cdots,n}.
$$
If we choose $j(x)=\max\{0,x\} \in J_0$ (as defined in~\eqref{class J0}), we get that
$$
\int_\Omega (u_{n,k;1} - u_{n,k;2} )^+ \,dv \leq \int_\Omega (u_{n,k-1;1} - u_{n,k-1;2} )^+ \, dv ,\quad   {\forall k=1,\cdots,n}.
$$

Let $u_{n;i}$, $i=1,\,2$, be the piecewise constant functions defined
by the $u_{n,k;i}$ as in~\eqref{piecewise constant fn}.  Then we have
that
\begin{equation} \label{eqn:approx-compare}
\int_\Omega (u_{n;1} - u_{n;2} )^+(t) \,dv \leq \int_\Omega (u_{0,1} - u_{0,2} )^+ \, dv,\quad   {\forall t\in (0,T]}. 
\end{equation}
Moreover, these
approximate solutions, if we pass to a subsequence,  converge in norm and pointwise
$v$-a.e. (see \eqref{mild sol def and converg}) to the mild solutions
$u_1,\,u_2$. 
More precisely, for any $t\in (0,T]$, a subsequence of
$\{(u_{n;1}-u_{n;2})^+(t)\}_{n=1}^\infty$ converges $v$-a.e. to $(u_1-u_2)^+(t)$. 
Then \eqref{comparison principle}  follows immediately from
Fatou's lemma and~\eqref{eqn:approx-compare}. 

}


\section{Asymptotic behavior of weak solutions}\label{Section:strong-asymptotics}

In this section we will prove Theorems~\ref{Thm:strong asymptotics}
and~\ref{Thm: finite time extinction}.  For brevity, throughout this
section we will write $\sigma$ for $\sigma_p$, the gain in the Sobolev
inequality in Hypothesis~\ref{hyp:sobolev}.


\subsection{Ultracontractive bounds}\label{Section: Decay estimate}
In this subsection we prove Theorem~\ref{Thm:strong asymptotics}. 
First, we claim that in order to show \eqref{asymptotics L2}, it
suffices to consider nonnegative initial data $u_0$.  This follows
from the comparison principle~\eqref{comparison principle}.  
Given any $u_0\in L^2_v(\Omega)$, let $u$ be  the solution of
\eqref{p-Lap-Dirichlet} with initial datum $u_0$; this exists by Theorem~\ref{Thm: strong solution}. 
To apply~\eqref{comparison principle} more easily, let   $u_1=u$,
$u_{0,1}=u_0$. Define $u_{0,2}=\max\{u_0,0 \}$ and let $u_2$ be the
solution of \eqref{p-Lap-Dirichlet} with initial datum $u_{0,2}$.
Then by \eqref{comparison principle}, we have that $u (t) = u_1(t) \leq u_2 (t)$ $v$-a.e. for all $t\geq 0$.
Similarly, if we let $u_3$ is the solution  of \eqref{p-Lap-Dirichlet}
with initial datum $\min\{u_0,0\}$, then  $u_3(t)  \leq u (t)$ $v$-a.e. for all $t\geq 0$. 
It then follows that  since~\eqref{asymptotics L2}  holds for $u_2$
and $u_3$, it also holds for $u$.

\medskip

Hereafter, we will assume $u_0$ is non-negative.  
It then follows from \eqref{comparison principle} that $u\geq 0$ $v$-a.e.
We first consider the case $u_0\in \HQp(\Omega)\cap L^\infty_v(\Omega)$.
Let   $u$ be the strong solution to \eqref{p-Lap-Dirichlet}. 
The proof for the smoothing effect of  $u$ is based on the following logarithmic Sobolev inequality.

\begin{lem}\label{Lem: Log Sob}
Given $f\in \HQp(\Omega)$, $r\in [1,\sigma p)$, and $\varepsilon>0$,
we have that
$$
\int_\Omega \frac{|f|^r}{\|f\|_{L^r_v}^r} \log \left( \frac{|f| }{\|f\|_{L^r_v} } \right) \, dv  \leq  \frac{\sigma}{\sigma p -r}\left(  \varepsilon { M_p^p} \frac{\| \nabla f\|_{\cL^p_Q}^p }{ \|f\|_{L^r_v}^p  } -\log \varepsilon  \right).
$$
\end{lem}
\begin{proof}
{Consider 
\begin{equation*}
\int_\Omega \frac{|f|^r}{\|f\|_{L^r_v}^r} \log \left( \frac{|f|
  }{\|f\|_{L^r_v} } \right) \, dv
=
\frac{1}{\sigma p-r} \int_\Omega \frac{|f|^r}{\|f\|_{L^r_v}^r} \log \left( \frac{|f|^{\sigma p-r} }{\|f\|_{L^r_v}^{\sigma p-r} } \right) \, dv. 
\end{equation*}
We now use Jensen's inequality with the probability measure
$$d\mu=\frac{|f|^r}{\|f\|_{L^r_v}^r}dv$$
to get
$$\int_{\Omega}\log \left( \frac{|f|^{\sigma p-r} }{\|f\|_{L^r_v}^{\sigma p-r} } \right) \,d\mu \leq \log\left(\int_{\Omega}  \frac{|f|^{\sigma p-r} }{\|f\|_{L^r_v}^{\sigma p-r} } \,d\mu\right)=\log \left( \frac{\|f\|_{L^{\sigma p}_v}^{\sigma
       p}}
   { \|f\|_{L^r_v}^{\sigma p} } \right).$$
By inequality \eqref{Sobolev ineq} and since $\log(x)\leq x$ for
$x\geq 0$,  we have
   
   \begin{multline*}
\int_\Omega \frac{|f|^r}{\|f\|_{L^r_v}^r} \log \left( \frac{|f|
  }{\|f\|_{L^r_v} } \right) \, dv\leq
 \frac{1}{\sigma p-r} \log \left( \frac{\|f\|_{L^{\sigma p}_v}^{\sigma
       p}}
   { \|f\|_{L^r_v}^{\sigma p} } \right) \leq \frac{\sigma p}{\sigma p -r }\log \left(M_p \frac{\| \nabla f\|_{\cL^p_Q} }{ \|f\|_{L^r_v}  } \right)
\\  \leq
  \frac{\sigma}{\sigma p -r}\left(  \varepsilon       M_p^p
    \frac{\| \nabla f\|_{\cL^p_Q}^p }{ \|f\|_{L^r_v}^p  } -\log \varepsilon  \right).
\end{multline*}}
\end{proof}
 
\begin{remark}
 We note that ultracontractive bounds for the
  $p$-Laplacian has also been proved using an argument that does not
  involve  the Logarithm Sobolev inequality. See \cite{Porzio09, Porzio15}.
\end{remark}


\begin{prop}\label{Prop: time der Lr norm}
  Given any $r\geq 1$, we have that
\begin{equation} \label{eqn:time-Lr}
  \frac{d}{dt} \|u(t)\|_{L^r_v}^r \leq
  -r(r-1) \left( \frac{p}{r+p-2} \right)^p \int_\Omega
  \left|\sqrt{Q( x)} \nabla   \left( (u(t,x))^{\frac{r+p-2}{p}} \right) \right|^p \, dx  
\end{equation}
for a.e. $t>0$.
\end{prop}

\begin{proof}
By the absolute continuity of $u^r$, we have that
\begin{align*}
   u^r(t+h,x)  - u^r(t,x)    
=   \int_t^{t+h} \partial_s  u^r(s,x)  \, ds  
=  r \int_t^{t+h} ( \partial_s u  u^{r-1} ) (s,x)\, ds . 
\end{align*}
By \eqref{L infty contraction}, 
$(\partial_s u  u^{r-1} )(s,x) $ is absolutely integrable on $([t,t+h], dt)\times (\Omega, dv)$. It follows from Fubini's Theorem that
\begin{align*}
 \frac{1}{h}\int_\Omega (u^r(t+h,x)  -  u^r(t,x)  )\, dv  
&=  \frac{r}{h} \int_\Omega \int_t^{t+h} (\partial_s u   u^{r-1} )(s,x) \, ds \, dv \\
&=  \frac{r}{h}\int_t^{t+h} \int_\Omega  (\partial_s u  u^{r-1} )(s,x) \, dv \, ds  
\to   r \int_\Omega  (\partial_t u  u^{r-1} )(t,x) \, dv
\end{align*}
for a.e. $0<t<T$ as $h\rightarrow 0$.  The limit follows from the
Lebesgue Differentiation Theorem.  Therefore,
\begin{equation}
\label{time der}
\frac{d}{dt} \int_\Omega   u^r(t,x) \, dv= r\int_\Omega  (\partial_t u  u^{r-1} )(t,x)\, dv.
\end{equation}

\medskip

Suppose first that $r\geq 2$; then by \eqref{L infty contraction} we
have that $u^{r-1} \in \HQp(\Omega)$ and
$\nabla (u^{r-1})=(r-1)u^{r-2} \nabla u $. {Since $u$
  is a solution on $(0,T)$ it is a solution on $[t,t+h]$ for any $t\in
  (0,T)$.  {Moreover,   when $r\geq 2$,}
$u^{r-1}$ is a valid test function in \eqref{Strong formulation} on
this interval.
Therefore, if we let $\phi= u^{r-1}$ in \eqref{Strong formulation} and
take the limit as $h\rightarrow 0$, then by the Lebesgue
Differentiation Theorem, we get
\begin{align*}
\frac{d}{dt} \|u(t)\|_{L^r_v}^r =&   r \int_\Omega \partial_t u(t,x) u^{r-1}(t,x) \, dv \\
=&   -r(r-1)\int_\Omega u^{r-2}(t,x)   |\sqrt{Q(t,x)} \nabla u(t,x)|^p \, dx \\
=& -r(r-1) \left( \frac{p}{r+p-2} \right)^p \int_\Omega    \left|\sqrt{Q( x)} \nabla   \left( (u(t,x))^{\frac{r+p-2}{p}} \right) \right|^p \, dx.    
\end{align*}
}
Now suppose that $1<r<2$.  For this case we have to  use an approximation argument. 
For every $n\in \N$, let $u_n=u+1/n$. 
Then by Proposition~\ref{Prop: chain rule A 2}, $\displaystyle
u_n^{r-1} - (1/n)^{r-1} $ is  a valid test function in \eqref{Strong
  formulation}.  
If we argue as we did to prove~\eqref{time der}, we have that
\begin{multline*}
 \frac{d}{dt} \left[ \int_\Omega  u_n^r\, dv  -  \left(\frac{1}{n}
  \right)^{r-1}   \int_\Omega u_n \, dv \right] 
=
\int_\Omega \partial_t u \left[ r u_n^{r-1} - \left(\frac{1}{n}
  \right)^{r-1} \right] \, dv   \\
=
- r \int_\Omega |\sqrt{Q} \nabla u_n|^{p-2}  Q \nabla u_n \cdot \nabla  u_n^{r-1}   \,   dx
=
- r(r-1) \left( \frac{p}{r+p-2} \right)^p \int_\Omega
\left|\sqrt{Q } \nabla  \left( u_n^{\frac{r+p-2}{p}} \right) \right|^p \, dx  .
\end{multline*}
In the second equality we have used Proposition~\ref{Prop: chain rule A}.

By the dominated convergence theorem,
$(u_n)^{\frac{r+p-2}{p}}\to u^{\frac{r+p-2}{p}}$ in $L^2_v(\Omega)$;
hence, by Proposition~\ref{Prop: m-accretive} we have that
$$
\int_\Omega     \left|\sqrt{Q } \nabla  \left( u^{\frac{r+p-2}{p}} \right) \right|^p \, dx \leq \liminf\limits_{n\to \infty}\int_\Omega     \left|\sqrt{Q } \nabla  \left( u_n^{\frac{r+p-2}{p}} \right) \right|^p \, dx.
$$
Again by the dominated convergence theorem, 
$$
\lim\limits_{n\to \infty} \int_\Omega \partial_t u   u_n^{r-1}\, dv=
\int_\Omega \partial_t u   u^{r-1}\, dv= \frac{1}{r}
\frac{d}{dt}\|u\|_{L^r_v}^r;
$$
moreover, it is immediate that
$$
\lim\limits_{n\to \infty} \left(\frac{1}{n} \right)^{r-1} \int_\Omega \partial_t u \, dv= 0.
$$
If we combine these estimate we get~\eqref{eqn:time-Lr}.

\medskip

Finally, if  $r=1$,~\eqref{eqn:time-Lr}  follows from  the dominated convergence theorem and \eqref{L infty contraction}:
$$
\frac{d}{dt} \|u(t)\|_{L^1_v}
=\int_\Omega \partial_t u \, dv = \lim\limits_{r\to 1^+} \int_\Omega
\partial_t u u^{r-1}\, dv
\leq 0.
$$  
\end{proof}

\medskip

We can now prove Theorem~\ref{Thm:strong asymptotics}. Recall that in
our hypotheses, we fixed $q_0$ such that
\begin{equation}
\label{cond q0}
q_0 \geq 1 \quad \text{and} \quad q_0>q_c:=\sigma'(2-p),
\end{equation}
where $\sigma'$ is the H\"older conjugate of $\sigma$.
Fix $t\in (0,\infty)$ and let $q:[0,t)\to [q_0,\infty)$ 
be a {\color{teal}strictly increasing}  $C^1$ function such that $q(0)=q_0$ and
$\lim\limits_{s\to t^-}q(s)=\infty$.

{ We first consider the special case when $u_0 \in L^\infty_v(\Omega) \cap
\HQp(\Omega) \subset L^2_v(\Omega) \cap
\HQp(\Omega)$.  By Theorem~\ref{Thm: strong solution} let $u$ be the
associated strong solution of~\eqref{p-Lap-Dirichlet}.}
Set $Y(s)=\log \|u(s)\|_{L^{q(s)}_v}$ and $b(s)=q(s)+p-2$.
Define   the weighted Young functional $J_v:[1,\infty)\times X \to \R$, where $X=L^1_v(\Omega)=\bigcap\limits_{h\geq 1} L^h_v(\Omega)$, by
$$
J_v(h,f)= \int_\Omega  \frac{|f |^h }{ \|f\|_{L^h_v}^h} \log\frac{|f |}{\|f \|_{L^h_v} } \, dv  .
$$
Note that $h J_v(h,f)=J_v(1,|f|^h)$.
{
One can apply Lemma~\ref{Lem: Log Sob} and Proposition~\ref{Prop: time der Lr norm} to derive that}
\begin{align}
\notag
&\frac{d}{ds} Y(s)  \\
\notag
& \qquad = - \frac{\dot{q}(s)}{q (s)} \log \|u(s)\|_{L^{q(s)}_v}  
+ \frac{1}{q(s) \|u(s)\|_{L^{q(s)}_v}^{ q(s)}} \frac{d}{ds} \|u(s)\|_{L^{q(s)}_v}^{q(s)} \\
\notag
&  \qquad \leq - \frac{\dot{q}(s)}{q (s)} \log \|u(s)\|_{L^{q(s)}_v}   + \frac{\dot{q}(s)}{q(s) \|u(s)\|_{L^{q(s)}_v}^{ q(s)}}\int_\Omega (u(s,x))^{q(s)} \log (u(s,x)) \, dv  \\
\label{mul der}
&\qquad \qquad - \frac{q(s)-1 }{  \|u(s)\|_{L^{q(s)}_v}^{ q(s)}} \left( \frac{p}{b(s)} \right)^p \int_\Omega    \left|\sqrt{Q( x)} \nabla   \left( (u(s,x))^{\frac{b(s)}{p}}\right) \right|^p \, dx \\
\notag
&\qquad  {\color{teal}\leq } \frac{\dot{q}(s)}{q^2(s)  } J_v(1, (u(s))^{q(s)}) -(q(s)-1 )\left(\frac{p}{b(s)} \right)^p \frac{\|u(s)\|_{L^{b(s)r(s)/p}_v}^{ b(s)}}{\|u(s)\|_{L^{q(s)}_v}^{ q(s)}} \frac{\log \varepsilon}{\varepsilon M_p^p}  \\
\notag
& \qquad  \qquad - {\frac{(q(s)-1) (\sigma p -r(s))}{ r(s)  M_p^p \varepsilon \sigma } }\left(\frac{p}{b(s)} \right)^p \frac{\|u(s)\|_{L^{b(s) r(s)/p}_v}^{ b(s)}}{\|u(s)\|_{L^{q(s)}_v}^{ q(s)}} J_v(1,(u(s))^{ b(s)r(s)/p}),
\end{align}
where $r(s)\in [1,\sigma p)$ will be fixed below, and $\dot{q}(s)=\frac{d}{ds}q(s)$.
{
When $q(s) \geq 2$, it follows from a similar proof of \cite[Lemma~2.18]{DavidScott21} that $u(s)^{\frac{b(s)}{p}} \in \HQp(\Omega)$.
When $q(s)<2$, the use of Proposition~\ref{Prop: time der Lr norm} above can be validated as follows. 
Let $f= u(s)^{\frac{b(s)}{p}}$ and $f_n= \left(u(s) + \frac{1}{ n}\right)^{\frac{b(s)}{p}} - \left(   \frac{1}{ n}\right)^{\frac{b(s)}{p}}$. It follows from Proposition~\ref{Prop: chain rule A 2} that $f_n\in \HQp(\Omega)$. Note that $f_n\to f$ a.e. and $|f_n | \leq M$  for some uniform constant $M>0$. 
Proposition~\ref{Prop: time der Lr norm} then implies that
\begin{align}
\notag
& \quad \int_\Omega    \left|\sqrt{Q( x)} \nabla   \left( (u(s,x))^{\frac{b(s)}{p}}\right) \right|^p \, dx  = \| \nabla f \|_{\cL_Q^p}^p \geq   \limsup\limits_{n\to \infty} \| \nabla f_n \|_{\cL_Q^p}^p  \\
\notag
& \geq    \limsup\limits_{n\to \infty}  \left( \frac{  \sigma p - r(s) }{ \sigma    \varepsilon M_p^p } \| f_n \|_{L^{r(s)}_v}^p \int_\Omega \frac{|f_n|^r}{\|f_n\|_{L^{r(s)}_v}^{r(s)}} \log \left( \frac{|f_n|}{\| f_n \|_{L^{r(s)}_v}}  \right) \, dv  + \frac{\log \varepsilon}{\varepsilon M_p^p } \| f_n \|_{L^{r(s)}_v}^p \right)\\
\label{Sobolev ineq app}
&=    \left( \frac{  \sigma p - r(s) }{ \sigma    \varepsilon M_p^p } \| f  \|_{L^{r(s)}_v}^p \int_\Omega \frac{|f |^r}{\|f  \|_{L^{r(s)}_v}^{r(s)}} \log \left( \frac{|f |}{\| f  \|_{L^{r(s)}_v}}  \right) \, dv   + \frac{\log \varepsilon}{\varepsilon M_p^p } \| f  \|_{L^{r(s)}_v}^p \right) \\
\notag 
& =  \|u(s)\|_{L^{b(s)r(s)/p}_v}^{ b(s)}  \left(\frac{  \sigma p - r(s) }{ r(s) \sigma    \varepsilon M_p^p }   J_v(1,(u(s))^{ b(s)r(s)/p}) + \frac{\log \varepsilon}{\varepsilon M_p^p }  \right).
\end{align}
We have used the dominated convergence theorem in \eqref{Sobolev ineq app}.
}
{
To derive \eqref{mul der}, we define $g(r,s):= \| u(s)\|_{L^r_v}^r$. By the dominated convergence theorem, we have
\begin{align*}
\partial_r g(r,s)  =  \partial_r \int_{\Omega} (u(s,x))^r \, dv = \int_\Omega  \partial_r e^{r \log (u(s,x))}\, dv  
 =  \int_\Omega (u(s,x))^r \log ( u(s,x)) \, dv.
\end{align*}
Then Proposition~\ref{Prop: time der Lr norm} implies that
\begin{align*}
\frac{d}{ds} \| u(s)\|_{L^{q(s)}_v}^{q(s)} & = \frac{d}{ds} g(q(s),s)\\
&= \dot{q}(s) \partial_r g(q(s),s) + \partial_s g(q(s),s) \\
&\leq \dot{q}(s)   \int_\Omega (u(s,x))^{q(s)} \log ( u(s,x)) \, dv\\
 &\quad -q(s)(q(s)-1) \left( \frac{p}{b(s) } \right)^p \int_\Omega
  \left|\sqrt{Q( x)} \nabla   \left( (u(s,x))^{\frac{b(s) }{p}} \right) \right|^p \, dx  .
\end{align*}
}
%
If we let
$$
\varepsilon=  \frac{q^2(s)  }{\dot{q}(s)}{ \frac{(q(s)-1) (\sigma
  p-r(s)) }{  M_p^p  \sigma  r(s)} }
\left(\frac{p}{b(s)} \right)^p\frac{\|u(s)\|_{L^{b(s)r(s)/p}_v}^{ b(s)}}{\|u(s)\|_{L^{q(s)}_v}^{ q(s)}},
$$
then we get
\begin{equation*}
  \frac{d}{ds} Y(s)    
  \leq   \frac{\dot{q}(s)}{q^2(s)  }[J_v(1,(u(s))^{q(s)}) - J_v(1,(u(s))^{b(s)r(s)/p})] \\
-\frac{\dot{q}(s)}{q^2(s)  } \frac{\sigma r(s)}{\sigma p-r(s)}     \log \varepsilon(s)   .
\end{equation*}
If we fix
$$
r(s)=\frac{p q(s)}{b(s)},
$$ 
then we have that
\begin{equation*}
 \frac{d}{ds} Y(s)  
  \leq   -\frac{\dot{q}(s)}{q(s)  }\ \frac{\sigma }{\sigma b(s)-q(s)}
  \log\left[  \frac{q(s)}{\dot{q}(s)}
    \frac{ (\sigma b(s)-q(s))(q(s)-1)}{M_p^p \sigma} \left(
      \frac{p}{b(s)}\right)^p   \right]
   - \frac{\dot{q}(s)}{q(s)  }\ \frac{\sigma(p-2)}{\sigma b(s)-q(s)} Y(s) .
 \end{equation*}

Observe that Condition~\eqref{cond q0} implies that $r(s)\in [1,\sigma p)$.
Now set
\begin{align*}
d(s) & =  \frac{\dot{q}(s)}{q(s)  }\ \frac{\sigma(p-2)}{\sigma b(s)-q(s)}, \\
h(s) & = \frac{\dot{q}(s)}{q(s)  }\ \frac{\sigma }{\sigma b(s)-q(s)} 
\,\log\left[  \frac{q(s)}{\dot{q}(s)} \frac{ (\sigma b(s)-q(s)
    )(q(s)-1)}{M_p^p \sigma }
  \left( \frac{p}{b(s)}\right)^p   \right] .
\end{align*}
Then $Y(s)$ satisfies the following differential inequality,
$$
\frac{d}{ds} Y(s) +d(s)Y(s) +h(s) \leq 0, \quad Y(0)=\log\|u_0\|_{{L^{q_0}_v}}.
$$
Hence, $Y(s) \leq Y_L(s)$, where $Y_L(s)$ satisfies the ordinary differential equation
$$
\frac{d}{ds} Y_L(s) +d(s)Y_L(s) +h(s) = 0, \quad Y_L(0)=\log\|u_0\|_{{L^{q_0}_v}}.
$$ 
The solution to this equation is
$$
Y_L(s)=e^{-\int_0^s d(a)\, da}\left[Y(0) - \int_0^s h(a) e^{\int_0^a d(\tau) \, d\tau} \, da \right]=: e^{-D(s)}\left[Y(0) - H(s) \right].
$$
If we set $q(s)=q_0t/(t-s)$, then by a direct computation we have that
\[ 
D(s) = \int_0^s d(a)\, da = \int_0^s \frac{\dot{q}(a)}{q(a)  }\
\frac{\sigma(p-2)}{\sigma(q(a)+p-2)-q(a)}\,da
=   \log \left[\frac{ \sigma'(p-2)t + q_0 t  }{ \sigma'(p-2)(t-s) +   q_0 t } \right];
\]
Therefore, if we take the limit, we get
$$
\lim\limits_{s\to t^-} e^{-D(s)}= \frac{ q_0}{\sigma'(p-2)  + q_0}.
$$

We now rewrite $H(s)$ as
\begin{align*}
H(s)= \frac{[\sigma (p-2)+ (\sigma-1) q_0]\sigma}{q_0} \sum\limits_{k=1}^3 H_k(s),
\end{align*}
where
\begin{align*}
H_1(s) = & \log t \int_0^s \frac{\dot{q}(a)}{[\sigma b(a)- q(a)]^2}\, da, \\
H_2(s) = & \log \left[ \frac{q_0 p^p}{{ M_p^p} \sigma p}
           \right]
           \int_0^s \frac{\dot{q}(a)}{[\sigma b(a)- q(a)]^2}\, da, \\
  H_3(s) =& \int_0^s \frac{\dot{q}(a)}{[\sigma b(a)- q(a)]^2}
            \log\left[ \frac{(q(a)-1) {\color{teal}(\sigma p b(a)-p q(a)) }}{q(a) b(a)^p   } \right]\, da.
\end{align*}
A direct computation shows that
$$
\lim\limits_{s\to t^-}\int_0^s \frac{\dot{q}(a)}{[\sigma b(a)-q(a)]^2}\, da
=\frac{1}{(\sigma-1)[\sigma (p-2)+(\sigma-1) q_0]}, 
$$
and Condition~\eqref{cond q0} ensures that
\begin{equation*}
 \lim\limits_{s\to t^-} H_3(s)
  = \int_{q_0}^\infty \frac{1}{ [(\sigma-1) y+ \sigma(p-2)]^2}
     \log\left[ \frac{(y-1)[p(\sigma-1) y+ \sigma p(p-2)]}{y (y+p-2)^p } \right]\, d y \\
=R_1, 
\end{equation*}
where $R_1$ is a constant depending only on $p$, $q_0$ and $\sigma$. 

Therefore, we have shown that
\begin{equation*}
  Y_L(t)  
= \frac{q_0}{\sigma'(p-2)  + q_0} Y(0) -\frac{\sigma'}{\sigma'(p-2)  + q_0} \log t +R_2  
\end{equation*}
for some constant $R_2$ depending only on $p$, $q_0$, $M_p$ and $\sigma$;
thus,
\begin{align*}
\log \| u (t)\|_{L^{\infty}_v}=& \lim\limits_{s\to t^-} \log \| u(s)\|_{L^{q(s)}_v}= \lim\limits_{s\to t^-}Y(s)\leq \lim\limits_{s\to t^-}Y_L(s)=Y_L(t) .
\end{align*}
If we now exponentiate both sides and rearrange terms, we get
 \begin{equation}
\label{asymptotics}
\| u (t)\|_{L^{\infty}_v} \leq C(\sigma,p,q_0,M_p) \frac{\|u_0\|_{L^{q_0}_v}^\gamma}{t^\beta} ,
\end{equation}
where  $\beta$ and $\gamma$ satisfy \eqref{Def: beta gamma}.

\medskip

Let $p_0=\max\{q_0,2\}$.  Given $ u_0 \in L^{p_0}_v(\Omega)$, fix a
sequence of functions $u_{0,j}$ in   $L^\infty_v(\Omega) \cap
\HQp(\Omega)$  
such that $u_{0,j} \to u_0$ in
$L^{p_0}_v(\Omega)$ and thus in $L^{q_0}_v(\Omega)$. Denote the
corresponding strong solution to \eqref{p-Lap-Dirichlet} with initial
data $u_{0,j}$ by $u_j$.  By \eqref{Strong converg u 1}, $u_j$
converges to the unique weak solution $u$ of \eqref{p-Lap-Dirichlet}
in $C([0,T];L^2_v(\Omega))$.  Therefore, since  by \eqref{asymptotics}
we have that
\begin{equation*}
\| u_j(t)\|_{L^\infty_v} \leq C  \frac{\|   u_{0,j}  \|_{L^{q_0}_v}^\gamma}{t^\beta}  ,
\end{equation*}
we can take the limit as $j\to \infty$ to get~\eqref{asymptotics L2}.  This completes
the proof of Theorem~\ref{Thm:strong asymptotics}.

\subsection{Finite Time Extinction for $1<p<2$}\label{Section: Finite time extinction}

In this subsection we prove Theorem~\ref{Thm: finite time extinction}.
Arguing as we did at the beginning of Section~\ref{Section: Decay estimate}, it suffices to consider nonnegative initial data $u_0$.
Recall $\displaystyle q_c:=\sigma'(2-p)$.  We first consider the case
when $u_0\in \HQp(\Omega)\cap L^{m }_v(\Omega)$, where
$m =\max\{q_c,2\}.$


First suppose that $q_c \geq 1$.   Set $\displaystyle h=\frac{q_c+p-2}{p}$ and define  $Y(t)=\|u(t)\|_{L^{q_c}_v}^{q_c}$. 
%
It follows from  Proposition~\ref{Prop: time der Lr norm} that
{
\begin{equation*}
  \frac{d}{dt} Y (t) \leq  -q_c(q_c-1) h^{-p} \int_\Omega
  |\sqrt{Q(x) } \nabla \big(  u(t,x)^h\big) |^p \, dx  \\
\leq    - \frac{ q_c(q_c-1)}{M_p^p h^p}  Y (t)^{ { 1/ \sigma  } } ,
\end{equation*}
}
where the last step follows from \eqref{Sobolev ineq gain}.
If we integrate this differential inequality, we get
\begin{align*}
Y (t) 
\begin{cases}
\leq Y (0) \left[ 1- \frac{ q_c (q_c-1) t }{\sigma'M_p^p h^p (Y(0))^{1/\sigma'}} \right]^{\sigma'} ,\quad & 0<t<T_0, \\
=0 , & t\geq T_0,
\end{cases}
\end{align*}
where
$$
T_0=\frac{h^p\sigma'M_p^p }{  q_c (q_c-1)}\|u_0\|_{L^{q_c}_v}^{\frac{ q_c }{\sigma'}}.
$$

Now suppose $q_c<1$.  Fix $q_0\in (1,2)$ and let
$h_0=\frac{q_0+p-2}{p}$. Define $Y(t)=\|u(t)\|_{L^{q_0}_v}^{q_0}$. If
we repeat the above argument and, \eqref{Sobolev ineq gain},  and then apply
H\"older's inequality, {using the fact that $h_0p\sigma=\sigma(q_0+p-2)>q_0$,} we get {
\begin{equation*}
  \frac{d}{dt} Y (t) \leq  -q_0(q_0-1) {h_0}^{-p} \int_\Omega
  |\sqrt{Q(x) } \nabla \big(  u(t,x)^{h_0}\big) |^p \, dx  
\leq    - \frac{ q_0(q_0-1)}{M_p^p {h_0}^p (v(\Omega))^{\frac{q_0
      -q_c}{q_0\sigma'}}}
Y (t)^{(q_0+p-2)/q_0} .
\end{equation*}
}
The existence of a finite  extinction time now follows by an argument
similar to the one given above.

For a general  $u_0\in L^m_v(\Omega)$, let $u$ be the solution of \eqref{p-Lap-Dirichlet} with initial datum $u_0$.
We take a sequence of initial data $\{u_{0,j}\}_{j=1}^\infty \subset \HQp(\Omega)\cap L^m_v(\Omega)$ converging to $u_0$ in $L^m_v(\Omega)$. 
{The extinction time, $T_{0,j}$ for each solution $u_j$ of
\eqref{p-Lap-Dirichlet} with initial datum $u_{0,j}$ depends on
$\|u_{0,j}\|_{L^{q_c}_v}$, but by taking $j$ sufficiently large we may
assume this quantity is bounded by $2\|u_0\|_{L^{q_c}_v}$, independent
of $j$. Therefore, by \eqref{Strong converg u 1} the extinction
time, $T_0$, of $u$ exists and is finite.}

 


\section{Ultracontractive Bounds and the Sobolev Inequality}\label{Section:asymptotics and Sobolev inequality}

In this section, we prove Theorem~\ref{Thm: equivalence}.  Hereafter, we assume $p\in [2,\infty)$. 
%
Fix $u_0\in \HQp(\Omega)$ and let $u$ be the corresponding strong  solution   of \eqref{p-Lap-Dirichlet} on $[0,T]$. 
In the previous section we showed that the Sobolev inequality
\eqref{Sobolev ineq gain} implies \eqref{asymptotics L2} with $\beta$
and $\gamma$ defined by \eqref{Def: beta gamma}.  This gives us that
${\rm (ii)}$ implies ${\rm (i)}$. 

We will now show that ${\rm (i)} \Rightarrow {\rm (iii)} \Rightarrow {\rm (ii)}$.
Suppose~\eqref{asymptotics L2} holds and fix $q_0\in [1,2)$.
If we use $u$ as a test function in \eqref{Strong formulation}, and
argue as we did to prove~\eqref{time der} {using the
  Lebesgue Differentiation Theorem}, we get 
\begin{equation}\label{Time der L2 norm}
  \frac{d}{dt} \|u(t)\|_{L^2_v}^{2}
  =-2\| \nabla u (t) \|_{\cL^p_Q}^p
\geq -2\| \nabla u_0 \|_{\cL^p_Q}^p ;
  \end{equation}
 the last inequality follows from \eqref{L infty nabla u}.   If we integrate both sides of
\eqref{Time der L2 norm} on $[0,t]$, we get
\begin{equation}
\label{equivalence est 1}
\|u(t)\|_{L^2_v}^2 -  \|u_0\|_{L^2_v}^2 \geq  - 2t \|\nabla u_0 \|_{\cL^p_Q}^p.
\end{equation}
Since $u(t) \in L^\infty_v(\Omega)$, by H\"older's inequality
(cf.~\cite[Inequality (7.9)]{MR1814364}) we have that
$$
\|u(t)\|_{L^2_v}\leq  \|u(t)\|_{L^\infty_v}^{1-\theta} \|u(t)\|_{L^{q_0}_v}^\theta
$$
with $\theta=q_0/2$.   From \eqref{L infty contraction} we have that 
$$
\|u(t)\|_{L^{q_0}_v} \leq \|u_0\|_{L^{q_0}_v} .
$$
If we combine this with \eqref{asymptotics L2},   then we get
$$
\|u(t)\|_{L^2_v}^2 \leq C \frac{\|u_0\|_{L^{q_0}_v}^{q_0+\gamma (2-q_0)}}{t^{\beta(2-q_0)}}.
$$

If we now put this expression into \eqref{equivalence est 1}, we get that
$$
C \frac{\|u_0\|_{L^{q_0}_v}^{q_0+\gamma (2-q_0)}}{t^{\beta(2-q_0)}}+ 2t \|\nabla u_0 \|_{\cL^p_Q}^p \geq \|u_0\|_{L^2_v}^2
$$
holds for all $t\in (0,\infty)$. We minimize the left hand side by setting
$$
t= \bigg( \frac{C  \beta (2-q_0) \|u_0\|_{L^{q_0}_v}^{q_0+\gamma
    (2-q_0)} }
{2 \|\nabla u_0 \|_{\cL^p_Q}^p} \bigg)^{\frac{1}{(1+\beta(2-q_0))}}; 
$$
with this value of $t$ we can arrange terms to get
\begin{equation}
\label{G-N-ineq-proof}
\|u_0\|_{L^2_v}^2 \leq C \|u_0\|_{L^{q_0}_v}^{\frac{q_0+\gamma
    (2-q_0)}{ 1+\beta(2-q_0)}} \|\nabla u_0
\|_{\cL^p_Q}^{\frac{p\beta(2-q_0)}{ 1+\beta(2-q_0)}}.
\end{equation}
If we use the definitions of $\beta$ and $\gamma$ in \eqref{Def: beta
  gamma}, a straightforward computation shows that
$$
q_0+\gamma (2-q_0) + p\beta(2-q_0)= 2[1+\beta(2-q_0) ].
$$
Therefore, \eqref{G-N-ineq-proof} becomes~\eqref{G-N-ineq} and we see that
${\rm (i)} \Rightarrow {\rm (iii)}$.

Finally, given this interpolation inequality, it follows from
\cite[Theorem~3.1]{Bakry95} that for  $u \in \HQp(\Omega)$,
\begin{equation*}
\|u\|_{L^{\sigma p}_v} \leq M_p \|\sqrt{Q}\nabla u\|_p,
\end{equation*} 
so ${\rm (iii)}\Rightarrow {\rm (ii)}$.
This completes the proof.

\section*{Acknowledgements}
{The authors would like to express their gratitude to the anonymous
referees for their very careful review of our  the manuscript and for
identifying a number of typos and minor gaps in our arguments.}

\bibliographystyle{plain}
\bibliography{p-laplacian}

\begin{thebibliography}{10}

\bibitem{Amann20}
H.~Amann.
\newblock Population dynamics in hostile neighborhoods.
\newblock {\em Rend. Istit. Mat. Univ. Trieste}, 52:27--43, 2020.

\bibitem{MR2722295}
F.~Andreu-Vaillo, J.~M. Maz\'{o}n, J.~D. Rossi, and J.~J. Toledo-Melero.
\newblock {\em Nonlocal diffusion problems}, volume 165 of {\em Mathematical
  Surveys and Monographs}.
\newblock American Mathematical Society, Providence, RI; Real Sociedad
  Matem\'{a}tica Espa\~{n}ola, Madrid, 2010.

\bibitem{Bakry95}
D.~Bakry, T.~Coulhon, M.~Ledoux, and L.~Saloff-Coste.
\newblock Sobolev inequalities in disguise.
\newblock {\em Indiana Univ. Math. J.}, 44(4):1033--1074, 1995.

\bibitem{Barbu}
V.~Barbu.
\newblock {\em Nonlinear semigroups and differential equations in {B}anach
  spaces}.
\newblock Editura Academiei Republicii Socialiste Rom\^{a}nia, Bucharest;
  Noordhoff International Publishing, Leiden, 1976.
\newblock Translated from the Romanian.

\bibitem{Barbu2010}
V.~Barbu.
\newblock {\em Nonlinear differential equations of monotone types in {B}anach
  spaces}.
\newblock Springer Monographs in Mathematics. Springer, New York, 2010.

\bibitem{BenilanThesis}
P.~B\'enilan.
\newblock {\em Equations d'\'evolution dans un espace de Banach quelconque et
  applications}.
\newblock Thesis, Univ. Orsay. 1972.

\bibitem{BenCra81}
P.~B\'{e}nilan and M.~G. Crandall.
\newblock Regularizing effects of homogeneous evolution equations.
\newblock In {\em Contributions to analysis and geometry ({B}altimore, {M}d.,
  1980)}, pages pp 23--39. Johns Hopkins Univ. Press, Baltimore, Md., 1981.

\bibitem{BonGrillo03}
M.~Bonforte, F.~Cipriani, and G.~Grillo.
\newblock Ultracontractivity and convergence to equilibrium for supercritical
  quasilinear parabolic equations on {R}iemannian manifolds.
\newblock {\em Adv. Differential Equations}, 8(7):843--872, 2003.

\bibitem{BonGrillo05}
M.~Bonforte and G.~Grillo.
\newblock Ultracontractive bounds for nonlinear evolution equations governed by
  the subcritical {$p$}-{L}aplacian.
\newblock In {\em Trends in partial differential equations of mathematical
  physics}, volume~61 of {\em Progr. Nonlinear Differential Equations Appl.},
  pages 15--26. Birkh\"{a}user, Basel, 2005.
\newblock Corrected version available electronically from the publisher.

\bibitem{BonGrillo06}
M.~Bonforte and G.~Grillo.
\newblock Super and ultracontractive bounds for doubly nonlinear evolution
  equations.
\newblock {\em Rev. Mat. Iberoam.}, 22(1):111--129, 2006.

\bibitem{BonGrillo07}
M.~Bonforte and G.~Grillo.
\newblock Singular evolution on manifolds, their smoothing properties, and
  {S}obolev inequalities.
\newblock {\em Discrete Contin. Dyn. Syst.}, pages 130--137, 2007.

\bibitem{MR0805809}
S.~Chanillo and R.~L. Wheeden.
\newblock Weighted {P}oincar\'{e} and {S}obolev inequalities and estimates for
  weighted {P}eano maximal functions.
\newblock {\em Amer. J. Math.}, 107(5):1191--1226, 1985.

\bibitem{ChanWh}
S.~Chanillo and R.~L. Wheeden.
\newblock Harnack's inequality and mean-value inequalities for solutions of
  degenerate elliptic equations.
\newblock {\em Comm. Partial Differential Equations}, 11(10):1111--1134, 1986.

\bibitem{MR0901091}
F.~Chiarenza.
\newblock Harnack inequality and degenerate parabolic equations.
\newblock In {\em Nonlinear parabolic equations: qualitative properties of
  solutions ({R}ome, 1985)}, volume 149 of {\em Pitman Res. Notes Math. Ser.},
  pages 63--67. Longman Sci. Tech., Harlow, 1987.

\bibitem{MR1030914}
F.~Chiarenza and M.~Franciosi.
\newblock Quasiconformal mappings and degenerate elliptic and parabolic
  equations.
\newblock {\em Matematiche (Catania)}, 42(1-2):163--170, 1987.

\bibitem{MR0757584}
F.~Chiarenza and M.~Frasca.
\newblock Boundedness for the solutions of a degenerate parabolic equation.
\newblock {\em Applicable Anal.}, 17(4):243--261, 1984.

\bibitem{MR0772255}
F.~Chiarenza and R.~P. Serapioni.
\newblock Degenerate parabolic equations and {H}arnack inequality.
\newblock {\em Ann. Mat. Pura Appl. (4)}, 137:139--162, 1984.

\bibitem{MR0748366}
F.~Chiarenza and R.P. Serapioni.
\newblock A {H}arnack inequality for degenerate parabolic equations.
\newblock {\em Comm. Partial Differential Equations}, 9(8):719--749, 1984.

\bibitem{CipGri01}
F.~Cipriani and G.~Grillo.
\newblock Uniform bounds for solutions to quasilinear parabolic equations.
\newblock {\em J. Differential Equations}, 177(1):209--234, 2001.

\bibitem{MR1911765}
F.~Cipriani and G.~Grillo.
\newblock {$L^q\text{-}L^\infty$} {H}\"{o}lder continuity for quasilinear
  parabolic equations associated to {S}obolev derivations.
\newblock {\em J. Math. Anal. Appl.}, 270(1):267--290, 2002.

\bibitem{CraLig71}
M.~G. Crandall and T.~M. Liggett.
\newblock Generation of semi-groups of nonlinear transformations on general
  {B}anach spaces.
\newblock {\em Amer. J. Math.}, 93:265--298, 1971.

\bibitem{CruzIsraMoen}
D.~Cruz-Uribe, J.~Isralowitz, and K.~Moen.
\newblock Two weight bump conditions for matrix weights.
\newblock {\em Integral Equations Operator Theory}, 90(3):Paper No. 36, 31,
  2018.

\bibitem{DavidKabeVirginia}
D.~Cruz-Uribe, K.~Moen, and V.~Naibo.
\newblock Regularity of solutions to degenerate {$p$}-{L}aplacian equations.
\newblock {\em J. Math. Anal. Appl.}, 401(1):458--478, 2013.

\bibitem{DavidKabeScott}
D.~Cruz-Uribe, K.~Moen, and S.~Rodney.
\newblock Matrix {$\mathcal A_p$} weights, degenerate {S}obolev spaces, and
  mappings of finite distortion.
\newblock {\em J. Geom. Anal.}, 26(4):2797--2830, 2016.

\bibitem{DavidRios08}
D.~Cruz-Uribe and C.~Rios.
\newblock Gaussian bounds for degenerate parabolic equations.
\newblock {\em J. Funct. Anal.}, 255(2):283--312, 2008.

\bibitem{MR3261119}
D.~Cruz-Uribe and C.~Rios.
\newblock Corrigendum to ``{G}aussian bounds for degenerate parabolic
  equations'' [{J}. {F}unct. {A}nal. 255 (2) (2008) 283--312] [mr2419963].
\newblock {\em J. Funct. Anal.}, 267(9):3507--3513, 2014.

\bibitem{DavidScott21}
D.~Cruz-Uribe and S.~Rodney.
\newblock Bounded weak solutions to elliptic {PDE} with data in {O}rlicz
  spaces.
\newblock {\em J. Differential Equations}, 297:409--432, 2021.

\bibitem{DavidScottEmily18}
D.~Cruz-Uribe, S.~Rodney, and E.~Rosta.
\newblock Poincar\'{e} inequalities and {N}eumann problems for the
  {$p$}-{L}aplacian.
\newblock {\em Canad. Math. Bull.}, 61(4):738--753, 2018.

\bibitem{DavidScottEmily20}
D.~Cruz-Uribe, S.~Rodney, and E.~Rosta.
\newblock Global {S}obolev inequalities and degenerate {$p$}-{L}aplacian
  equations.
\newblock {\em J. Differential Equations}, 268(10):6189--6210, 2020.

\bibitem{DiB83}
E.~DiBenedetto.
\newblock {$C\sp{1+\alpha }$} local regularity of weak solutions of degenerate
  elliptic equations.
\newblock {\em Nonlinear Anal.}, 7(8):827--850, 1983.

\bibitem{DiB}
E.~DiBenedetto.
\newblock {\em Degenerate parabolic equations}.
\newblock Universitext. Springer-Verlag, New York, 1993.

\bibitem{DiBGiVe}
E.~DiBenedetto, U.~Gianazza, and V.~Vespri.
\newblock {\em Harnack's inequality for degenerate and singular parabolic
  equations}.
\newblock Springer Monographs in Mathematics. Springer, New York, 2012.

\bibitem{Evans82}
L.~C. Evans.
\newblock A new proof of local {$C\sp{1,\alpha }$} regularity for solutions of
  certain degenerate elliptic p.d.e.
\newblock {\em J. Differential Equations}, 45(3):356--373, 1982.

\bibitem{MR0688024}
E.B. Fabes, D.~Jerison, and C.~E. Kenig.
\newblock The {W}iener test for degenerate elliptic equations.
\newblock {\em Ann. Inst. Fourier (Grenoble)}, 32(3):vi, 151--182, 1982.

\bibitem{MR0730093}
E.B. Fabes, C.~E. Kenig, and D.~Jerison.
\newblock Boundary behavior of solutions to degenerate elliptic equations.
\newblock In {\em Conference on harmonic analysis in honor of {A}ntoni
  {Z}ygmund, {V}ol. {I}, {II} ({C}hicago, {I}ll., 1981)}, Wadsworth Math. Ser.,
  pages 577--589. Wadsworth, Belmont, CA, 1983.

\bibitem{MR0643158}
E.B. Fabes, C.~E. Kenig, and R.~P. Serapioni.
\newblock The local regularity of solutions of degenerate elliptic equations.
\newblock {\em Comm. Partial Differential Equations}, 7(1):77--116, 1982.

\bibitem{Ferrari}
F.~Ferrari.
\newblock Harnack inequality for two-weight subelliptic {$p$}-{L}aplace
  operators.
\newblock {\em Math. Nachr.}, 279(8):815--830, 2006.

\bibitem{Fichera60}
G.~Fichera.
\newblock On a unified theory of boundary value problems for elliptic-parabolic
  equations of second order.
\newblock In {\em Boundary problems in differential equations}, pages pp
  97--120. Univ. Wisconsin Press, Madison, Wis., 1960.

\bibitem{MR1814364}
D.~Gilbarg and N.~S. Trudinger.
\newblock {\em Elliptic {P}artial {D}ifferential {E}quations of {S}econd
  {O}rder}.
\newblock Classics in Mathematics. Springer-Verlag, Berlin, 2001.
\newblock Reprint of the 1998 edition.

\bibitem{MR1962933}
E.~Giusti.
\newblock {\em Direct methods in the calculus of variations}.
\newblock World Scientific Publishing Co., Inc., River Edge, NJ, 2003.

\bibitem{GuidottiShao17}
P.~Guidotti and Y.~Shao.
\newblock Wellposedness of a nonlocal nonlinear diffusion equation of image
  processing.
\newblock {\em Nonlinear Anal.}, 150:114--137, 2017.

\bibitem{MR2305115}
J.~Heinonen, T.~Kilpel\"{a}inen, and O.~Martio.
\newblock {\em Nonlinear potential theory of degenerate elliptic equations}.
\newblock Dover Publications, Inc., Mineola, NY, 2006.
\newblock Unabridged republication of the 1993 original.

\bibitem{HerreroVaz81}
M.~A. Herrero and J.~L. V\'{a}zquez.
\newblock Asymptotic behaviour of the solutions of a strongly nonlinear
  parabolic problem.
\newblock {\em Ann. Fac. Sci. Toulouse Math. (5)}, 3(2):113--127, 1981.

\bibitem{MR4224718}
L.~Korobenko, C.~Rios, E.~Sawyer, and R.~Shen.
\newblock Local boundedness, maximum principles, and continuity of solutions to
  infinitely degenerate elliptic equations with rough coefficients.
\newblock {\em Mem. Amer. Math. Soc.}, 269(1311):vii+130, 2021.

\bibitem{KosSotWang}
P.~Koskela, T.~Soto, and Z.~Wang.
\newblock Traces of weighted function spaces: dyadic norms and {W}hitney
  extensions.
\newblock {\em Sci. China Math.}, 60(11):1981--2010, 2017.

\bibitem{Lewis}
J.~L. Lewis.
\newblock Regularity of the derivatives of solutions to certain degenerate
  elliptic equations.
\newblock {\em Indiana Univ. Math. J.}, 32(6):849--858, 1983.

\bibitem{Lieber}
G.~M. Lieberman.
\newblock {\em Second order parabolic differential equations}.
\newblock World Scientific Publishing Co., Inc., River Edge, NJ, 1996.

\bibitem{Lions}
J.-L. Lions.
\newblock {\em Quelques m\'{e}thodes de r\'{e}solution des probl\`emes aux
  limites non lin\'{e}aires}.
\newblock Dunod, Paris; Gauthier-Villars, Paris, 1969.

\bibitem{MR0839035}
G.~Modica.
\newblock Quasiminima of some degenerate functionals.
\newblock {\em Ann. Mat. Pura Appl. (4)}, 142:121--143, 1985.

\bibitem{MR3922808}
D.~Monticelli, K.~Payne, and F.~Punzo.
\newblock Poincar\'e{} inequalities for {S}obolev spaces with matrix-valued
  weights and applications to degenerate partial differential equations.
\newblock {\em Proc. Roy. Soc. Edinburgh Sect. A}, 149(1):61--100, 2019.

\bibitem{MR3369270}
D.~D. Monticelli and S.~Rodney.
\newblock Existence and spectral theory for weak solutions of {N}eumann and
  {D}irichlet problems for linear degenerate elliptic operators with rough
  coefficients.
\newblock {\em J. Differential Equations}, 259(8):4009--4044, 2015.

\bibitem{MR3388872}
D.~D. Monticelli, S.~Rodney, and R.L. Wheeden.
\newblock Harnack's inequality and {H}\"{o}lder continuity for weak solutions
  of degenerate quasilinear equations with rough coefficients.
\newblock {\em Nonlinear Anal.}, 126:69--114, 2015.

\bibitem{MR2906551}
D.D. Monticelli, S.~Rodney, and R.~L. Wheeden.
\newblock Boundedness of weak solutions of degenerate quasilinear equations
  with rough coefficients.
\newblock {\em Differential Integral Equations}, 25(1-2):143--200, 2012.

\bibitem{Oleinik64}
O.~A. Ole\u{\i}nik.
\newblock On a problem of {G}. {F}ichera.
\newblock {\em Dokl. Akad. Nauk SSSR}, pages 1297--1300, 1964.

\bibitem{Oleinik66}
O.~A. Ole\u{\i}nik.
\newblock On linear equations of the second order with a non-negative
  characteristic form.
\newblock {\em Mat. Sb. (N.S.)}, pages 111--140, 1966.

\bibitem{PerRela}
C.~P\'{e}rez and E.~Rela.
\newblock Degenerate {P}oincar\'{e}-{S}obolev inequalities.
\newblock {\em Trans. Amer. Math. Soc.}, 372(9):6087--6133, 2019.

\bibitem{Porzio09}
M.~M. Porzio.
\newblock On decay estimates.
\newblock {\em J. Evol. Equ.}, 9(3):561--591, 2009.

\bibitem{MR2819280}
M.~M. Porzio.
\newblock Existence, uniqueness and behavior of solutions for a class of
  nonlinear parabolic problems.
\newblock {\em Nonlinear Anal.}, 74(16):5359--5382, 2011.

\bibitem{Porzio15}
M.~M. Porzio.
\newblock On uniform and decay estimates for unbounded solutions of partial
  differential equations.
\newblock {\em J. Differential Equations}, 259(12):6960--7011, 2015.

\bibitem{MR1350650}
A.~Ron and Z.~Shen.
\newblock Frames and stable bases for shift-invariant subspaces of
  {$L_2(\mathbf R^d)$}.
\newblock {\em Canad. J. Math.}, 47(5):1051--1094, 1995.

\bibitem{MR2574880}
E.~Sawyer and R.~L. Wheeden.
\newblock Degenerate {S}obolev spaces and regularity of subelliptic equations.
\newblock {\em Trans. Amer. Math. Soc.}, 362(4):1869--1906, 2010.

\bibitem{MR2204824}
E.~Sawyer and R.L. Wheeden.
\newblock H\"{o}lder continuity of weak solutions to subelliptic equations with
  rough coefficients.
\newblock {\em Mem. Amer. Math. Soc.}, 180(847):x+157, 2006.

\bibitem{Shao18}
Y.~Shao.
\newblock The {Y}amabe flow on incomplete manifolds.
\newblock {\em J. Evol. Equ.}, 18(4):1595--1632, 2018.

\bibitem{Showalter}
R.~E. Showalter.
\newblock {\em Monotone operators in {B}anach space and nonlinear partial
  differential equations}, volume~49 of {\em Mathematical Surveys and
  Monographs}.
\newblock American Mathematical Society, Providence, RI, 1997.

\bibitem{Simon}
J.~Simon.
\newblock R\'{e}gularit\'{e} de la solution d'une \'{e}quation non lin\'{e}aire
  dans {${\bf R}\sp{N}$}.
\newblock In {\em Journ\'{e}es d'{A}nalyse {N}on {L}in\'{e}aire ({P}roc.
  {C}onf., {B}esan\c{c}on, 1977)}, volume 665 of {\em Lecture Notes in Math},
  pages pp 205--227. Springer, Berlin, 1978.

\bibitem{StanVaz13}
D.~Stan and J.~L. V\'{a}zquez.
\newblock Asymptotic behaviour of the doubly nonlinear diffusion equation
  {$u_t=\Delta_pu^m$} on bounded domains.
\newblock {\em Nonlinear Anal.}, 77:1--32, 2013.

\bibitem{Tolksdorf}
P.~Tolksdorf.
\newblock Regularity for a more general class of quasilinear elliptic
  equations.
\newblock {\em J. Differential Equations}, 51(1):126--150, 1984.

\bibitem{Trudinger67}
N.~S. Trudinger.
\newblock On {H}arnack type inequalities and their application to quasilinear
  elliptic equations.
\newblock {\em Comm. Pure Appl. Math.}, 20:721--747, 1967.

\bibitem{Tyulenev}
A.~I. Tyulenev.
\newblock Description of traces of functions in the {S}obolev space with a
  {M}uckenhoupt weight.
\newblock {\em Proc. Steklov Inst. Math.}, 284(1):280--295, 2014.
\newblock Translation of Tr. Mat. Inst. Steklova {\bf 284} (2014), 288--303.

\bibitem{Tyulenev2}
A.~I. Tyulenev.
\newblock Traces of weighted {S}obolev spaces with {M}uckenhoupt weight. {T}he
  case {$p=1$}.
\newblock {\em Nonlinear Anal.}, 128:248--272, 2015.

\bibitem{Ural}
N.N. Ural'ceva.
\newblock Degenerate quasilinear elliptic systems.
\newblock {\em Zap. Nau\v{c}n. Sem. Leningrad. Otdel. Mat. Inst. Steklov.
  (LOMI)}, pages 184--222, 1968.

\bibitem{Vaz06book}
J.~L. V\'{a}zquez.
\newblock {\em Smoothing and decay estimates for nonlinear diffusion
  equations}, volume~33 of {\em Oxford Lecture Series in Mathematics and its
  Applications}.
\newblock Oxford University Press, Oxford, 2006.
\newblock Equations of porous medium type.

\bibitem{YinWang04}
J.~Yin and C.~Wang.
\newblock Properties of the boundary flux of a singular diffusion process.
\newblock {\em Chinese Ann. Math. Ser. B}, 25(2):175--182, 2004.

\bibitem{YinWang07}
J.~Yin and C.~Wang.
\newblock Evolutionary weighted {$p$}-{L}aplacian with boundary degeneracy.
\newblock {\em J. Differential Equations}, 237(2):421--445, 2007.

\bibitem{ZhanFeng19}
H.~Zhan and Z.~Feng.
\newblock Partial boundary value condition for a nonlinear degenerate parabolic
  equation.
\newblock {\em J. Differential Equations}, 267(5):2874--2890, 2019.

\bibitem{ZhanFeng21}
H.~Zhan and Z.~Feng.
\newblock Optimal partial boundary condition for degenerate parabolic
  equations.
\newblock {\em J. Differential Equations}, 284:156--182, 2021.

\end{thebibliography}

\end{document}